\documentclass[a4paper, 12pt]{amsart}
\usepackage{amsmath}
\usepackage{amssymb, latexsym, amsthm}
\usepackage[all]{xy}

\newtheorem{Thm}{Theorem}[section]
\newtheorem{Prop}[Thm]{Proposition}
\newtheorem{Cor}[Thm]{Corollary}
\newtheorem{Lem}[Thm]{Lemma}

\theoremstyle{definition}
\newtheorem*{definition}{Definition}
\newtheorem*{remark}{Remark}

\numberwithin{equation}{section}

\begin{document}

\newcommand{\Z}[0]{\mathbb{Z}}
\newcommand{\Q}[0]{\mathbb{Q}}
\newcommand{\F}[0]{\mathbb{F}}
\newcommand{\N}[0]{\mathbb{N}}
\renewcommand{\O}[0]{\mathcal{O}}
\newcommand{\p}[0]{\mathfrak{p}}
\newcommand{\m}[0]{\mathrm{m}}
\newcommand{\Tr}{\mathrm{Tr}}
\newcommand{\Hom}[0]{\mathrm{Hom}}
\newcommand{\Gal}[0]{\mathrm{Gal}}
\newcommand{\Res}[0]{\mathrm{Res}}
\newcommand{\id}{\mathrm{id}}
\newcommand{\cl}{\mathrm{cl}}
\newcommand{\mult}{\mathrm{mult}}
\newcommand{\adm}{\mathrm{adm}}
\newcommand{\tr}{\mathrm{tr}}
\newcommand{\pr}{\mathrm{pr}}
\newcommand{\Ker}{\mathrm{Ker}}
\newcommand{\ab}{\mathrm{ab}}
\newcommand{\sep}{\mathrm{sep}} 
\newcommand{\triv}{\mathrm{triv}}
\newcommand{\alg}{\mathrm{alg}}
\newcommand{\ur}{\mathrm{ur}}
\newcommand{\Coker}{\mathrm{Coker}}
\newcommand{\Aut}{\mathrm{Aut}}
\newcommand{\Ext}{\mathrm{Ext}}
\newcommand{\Iso}{\mathrm{Iso}}
\newcommand{\GL}{\mathrm{GL}}
\newcommand{\Fil}{\mathrm{Fil}}
\newcommand{\an}{\mathrm{an}}
\renewcommand{\c}{\mathcal }
\newcommand{\crys}{\mathrm{crys}}
\newcommand{\st}{\mathrm{st}}
\newcommand{\Map}{\mathrm{Map}}
\newcommand{\Sym}{\mathrm{Sym}}
\newcommand{\Spec}{\mathrm{Spec}}
\newcommand{\Frac}{\mathrm{Frac}}
\newcommand{\LT}{\mathrm{LT}}
\newcommand{\Alg}{\mathrm{Alg}}
\newcommand{\To}{\longrightarrow}
\newcommand{\Aug}{\mathrm{Aug}}
\newcommand{\wt}{\widetilde}
\newcommand{\op}{\mathrm}
\newcommand{\Ad}{\op{Ad}}
\newcommand{\ad}{\op{ad}}
\newcommand{\fr}{\mathfrak}

\title[Ramification via deformations]
{Ramification filtration via deformations}
\author{Victor Abrashkin}
\address{Department of Mathematical Sciences, 
Durham University, Science Laboratories, 
South Rd, Durham DH1 3LE, United Kingdom \ \&\ Steklov 
Institute, Gubkina str. 8, 119991, Moscow, Russia
}
\email{victor.abrashkin@durham.ac.uk}
\date{}

\begin{abstract} Let $\mathcal K$ be a field of formal Laurent series 
with coefficients in a finite field of 
characteristic $p$, $\mathcal G_{<p}$ --- the maximal quotient of 
the Galois group of $\c K$ 
of period $p$ and 
nilpotent class $<p$ and $\{\mathcal G_{<p}^{(v)}\}_{v\geqslant 0}$ --- 
the filtration by ramification subgroups in the upper numbering.  
Let $\mathcal G_{<p}=G(\mathcal L)$ be the identification of 
nilpotent Artin-Schreier theory: here 
$G(\mathcal L)$ is the group obtained from a suitable profinite Lie 
$\mathbb{F} _p$-algebra $\mathcal L$ via the Campbell-Hausdorff composition law. 
We develop a new technique to the description of the ideals 
$\mathcal L^{(v)}$ such that $G(\mathcal L^{(v)})=\mathcal G_{<p}^{(v)}$ 
and the explicite construction of their generators. 
Given $v_0\geqslant 1$ we construct epimorphism of Lie algebras 
$\bar\eta ^{\dag }:\mathcal L\longrightarrow 
\bar{\mathcal L}^{\dag }$ and an action 
$\Omega _U$ of the formal group of order $p$,  
$\alpha _p=\operatorname{Spec}\,\mathbb{F} _p[U]$, $U^p=0$, 
on $\bar{\mathcal L}^{\dag }$. Suppose  
$d\Omega _U=B^{\dag }U$, where 
$B^{\dag }\in\operatorname{Diff}\bar{\mathcal L}^{\dag }$, 
and 
 $\bar{\mathcal L}^{\dag }[v_0]$ is the ideal of 
$\bar{\mathcal L}^{\dag }$ generated 
by the elements of $B^{\dag }(\bar{\mathcal L}^{\dag })$. The main result of 
the paper states that 
$\mathcal L^{(v_0)}=(\bar\eta ^{\dag })^{-1}\bar{\mathcal L}^{\dag }[v_0]$. 
In the last sections we relate this 
result to the explicit construction of generators of $\mathcal L^{(v_0)}$ obtained 
earlier by the author, develop its more efficient version and apply it to the 
recovering of the whole ramification filtration of $\mathcal G_{<p}$ from 
the set of its jumps. 
\end{abstract}
\maketitle

\section*{Introduction} \label{S0}

Let $\c K$ be a complete discrete valuation field of characteristic $p$
with finite residue field $k\simeq\F _{p^{N_0}}$, $N_0\in\N $. 
Let $\c K_{<p}$ be a maximal $p$-extension of $\c K$ with the Galois group 
$\Gal (\c K_{<p}/\c K):=\c G_{<p}$ 
of nilpotence class $<p$ and exponent $p$. The advantage of $\c G_{<p}$ 
(compared to the whole Galois group $\c G$ of $\c K$) comes 
from the following fact:   
any $p$-group $G$ of nilpotence class $s_0<p$ and exponent 
$p$ can be presented in the form 
$G(L)$, where $L$ is a Lie $\F _p$-algebra of 
nilpotence class $s_0$ and the set $G(L):=L$ 
is provided with a group structure via the 
Campbell-Hausdorff composition law, cf.\,Sect.\,\ref{S1.2}.  

Consider the decreasing filtration by ramification subgroups in the upper numbering
$\{\c G^{(v)}_{<p}\}_{v\geqslant 0}$ of $\c G_{<p}$. 
This filtration substantially reflects arithmetic structure 
of the field $\c K$, cf. \cite{Ab10}. First results about the structure of these  
ramification subgroups were obtained by the 
author in \cite{Ab1}. This approach included: 

a) a construction of the 
identification $\c G_{<p}= G(\c L)$, where $\c L$ is explicitly defined 
Lie $\F _p$-algebra (nilpotent Artin-Shreier theory);

b) a construction of ideals $\c L^{(v)}$ such that 
$\c G_{<p}^{(v)}=G(\c L^{(v)})$.

Namely, we constructed 
explicit elements 
$\c F_{\alpha , -N}\in\c L\otimes k$, where $\alpha\geqslant 1$ and  
$N\in\Z _{\geqslant 0}$, 
allowing us to characterize the ideals $\c L^{(v)}$ as follows. 
Given $v_0\geqslant 1$ there is $N(v_0)\in\Z _{\geqslant 0}$ such that   
$\c L^{(v_0)}$ is the minimal ideal in $\c L$ satisfying the condition: 
if $\alpha \geqslant v_0$ and $N\geqslant N(v_0)$ then  
$\c F_{\alpha , -N}\in\c L^{(v_0)}\otimes k$. 

For a generalization of these results cf.\,\cite{Ab2, Ab3} and 
for their application to 
an analogue of the Grothendieck conjecture  cf.\,\cite{Ab4, Ab5}. For  
the study of an analogue $\Gamma _{<p}=G(L)$ of 
the group $\c G_{<p}$ in the case of 
local fields $K$ of mixed characteristic containing  
$p$-th roots of unity cf.\,\cite{Ab11, Ab12}. In these two papers 
we obtained a description of  
the corresponding ramification ideals $L^{(v)}$ and their interpretation 
in terms of the Demushkin relation for $\Gamma _{<p}$. 
Our method was based on a new technique (a linearization procedure) 
which allowed us to work with  
arithmetic properties of local fields in terms of Lie algebras. The statement  
of final results in terms of Lie algebras looked quite natural. We believe that it 
would be difficult to achieve such results exclusively via group theoretic means. 
To some extent this phenomenon could be treated 
as an evidence of the existence of a 
hidden ``analytic structure'' 
on the Galois group which shows up on the level 
of Lie algebras in our case. 
However, the above mentioned study of the mixed characteristic case  
is based quite substantially on the 
characteristic $p$ results from the papers \cite {Ab1}, \cite{Ab2} and \cite{Ab3}. 
It should be pointed out that  in \cite{Ab1} the  proof 
of the main result was not done 
completely in terms of Lie algebras. We could not \lq\lq linearize\rq\rq\ 
the verification of the criterion describing the ramification ideals 
$\c L^{(v)}$. As a result, we proceeded with non-trivial 
calculations in the enveloping algebra of $\c L$. 
In later papers \cite{Ab2} and \cite{Ab3} 
we managed to  
generalize our approach to the case of groups of period $p^M$, $M>1$ 
(but still of nilpotence class $<p$).  
At the same time it became clear that we should develop 
new techniques and methods when working with more complicated objects, e.g. 
higher local fields, cf.\,e.g.\,\cite{Ab13}. 

In this paper we develop a linearization procedure 
which allows us to obtain the results from 
\cite{Ab1}  exclusively in 
terms of Lie theory. For a given $v_0>0$, we characterize 
the ramification ideal $\c L^{(v_0)}$ in terms of deformations of some auxilliary 
Lie $\F _p$-algebra $\bar{\c L}^{\dag }$ with a suitably chosen 
module of coefficients. 
This algebra is provided with an action of a formal group 
of order $p$ which comes from 
a derivation of a higher order. The appearance of such 
derivations is quite a new phenomenon. Note that  in 
\cite{Ab11, Ab12} we also 
used the action of formal group of order $p$ 
but it came from usual derivations. 

Let us sketch briefly the main steps of our approach.

We start with a choice of an (sufficiently general) epimorphism 
$\eta _e:\c G\To G(\c L)$ which induces identification 
$\c G_{<p}\simeq G(\c L)$ given by the nilpotent Artin-Shreier theory.  
Here $\c L$ is a profinite Lie $\F _p$-algebra such that 
its extension of scalars $\c L_{k}:=\c L\otimes k$ 
has a fixed set of profinite generators. The map  
$\eta _{e}$ depends on a choice of an element 
$e\in\c L_{\c K}:=\c L\otimes\c K$ specified below. 
\medskip 

Choose $v_0\in\mathbb R$, $v_0>0$. 
We aim to characterize the ideal $\c L^{(v_0)}\subset\c L$ 
such that $\eta _e(\c G^{(v_0)})=G(\c L^{(v_0)})$. 
For this reason we: 
\medskip

a)\ define a decreasing central filtration of $\c L$ by its ideals 
$\c L=\c L(1)\supset \dots \supset \c L(s)\supset \dots \, ,$  
and set $\bar{\c L}=\c L/\c L(p)$ with the induced filtration 
$\{\bar{\c L}(s)\}_{s\geqslant 1}$ (note that $\bar{\c L}(p)=0$);  
\medskip  

b)\ introduce a lift $\c V:\bar{\c L}^{\dag }\To \bar{\c L}$ where 
$\bar{\c L}^{\dag }$ is a Lie $\F _p$-algebra of 
nilpotent class $<p$ together with its 
central filtration $\bar{\c L}^{\dag }(s)$ such that 
${\c V}(\bar{\c L}^{\dag }(s))=
\bar{\c L}(s)$ and $\bar{\c L}^{\dag }(p)=0$;
\medskip 

c)\ specify a group epimorphism  
$\eta _{\bar e^{\dag }}:\c G\To G(\bar{\c L}^{\dag })$ 
such that 
$${\c V}\eta _{\bar e^{\dag }}=
\eta _{\bar{e}}:=\eta _{e}\,\op{mod}\,G(\c L(p));$$

d)\ introduce the actions 
$\Omega _{\gamma }:\bar{\c L}^{\dag }\To\bar{\c L}^{\dag }$ of the elements 
$\gamma\in\Z /p$;
\medskip 

e)\ introduce the ideal $\bar {\c L}[v_0]$ in $\bar{\c L}$ as the minimal ideal 
such that for any $\gamma\in\Z /p$, 
$\c V^{-1}\bar{\c L}[v_0]\supset \Omega _{\gamma }(\Ker {\c V})$ 
(this condition is not easy to study because the action of $\Z /p$  
appears in terms of 
complicated 
Campbell-Hausdorff group law);
\medskip 

f)\ establish that the actions $\Omega _{\gamma }$ can 
be defined in terms of some co-action 
$\Omega _U:
\bar{\c L}^{\dag }\To \bar {\c L}^{\dag }\otimes \F_p[U]$ of 
the formal group scheme 
$\alpha _p=\F _p[U]$, $U^p=0$, with coaddition 
$\Delta U=U\otimes 1+1\otimes U$;
\medskip 

g)\  if $d\Omega _U=B^{\dag }U$ is the differential of $\Omega _U$ (here 
$B^{\dag }\in \op{Diff}\bar{\c L}^{\dag }$) 
then $\bar{\c L}[v_0]$ appears as the minimal ideal in $\bar{\c L}$ containing  
${\c V}B^{\dag }(\bar{\c L}^{\dag })$;
\medskip 

h)\ verify  
that $\c L^{(v_0)}=\bar{\op{pr}} ^{-1}\bar{\c L}[v_0]$, where $\bar{\op{pr}}$ is 
the natural projection from $\c L$ to $\bar{\c L}$. 
\medskip 

The above characterization 
of $\c L^{(v_0)}$ can be used for a considerable simplification 
of the process of recovering of explicit generators. 
These generators appeared in \cite{Ab1} as \lq\lq linear\rq\rq\ 
components of some elements from $\c L^{(v_0)}$. Our method allows us to 
skip the verification that these linear components generate the 
ideal $\c L^{(v_0)}$.  
\medskip 

In the final Section we relate the description of ramification ideals with 
their description in \cite{Ab1}, discuss the problem of effective 
construction of their generators, and show how the knowledge of 
the jumps of ramification filtration in $\c G_{<p}$ allows us 
to recover the structure of this filtration. 

The methods of this paper admit a generalization to the Galois groups of 
period $p^M$ as well as to the case 
of higher dimensional local fields in the 
characteristic $p$ case. 
In particular, the \lq\lq $p^M$-version\rq\rq\ \cite{Ab3} of \cite{Ab1} 
required much more complicated study of \lq\lq non-linear\rq\rq\ 
components,  
which can be now avoided due to our approach 
(the paper in preparation).
This also will provide us with much better background 
for the papers \cite{Ab11, Ab12} and their  
upcoming 
\lq\lq $p^M$-versions\rq\rq\ including the case of higher 
dimensional local fields.

{\bf Notation.} 
 Suppose $s\in\N $. For any topological 
group $G$, we denote by $C_s(G)$ the closure of the subgroup of 
$G$ generated by the commutators of order $\geqslant s$. If $L$ is a topological 
Lie algebra then $C_s(L)$ is the closure of the ideal generated by 
commutators of degree $\geqslant s$. 
For any topological $A$-modules $M$ and $B$ we use the notation  
$M_B:=M\hat\otimes _AB$.

\section{Preliminaries} \label{S1} 

Suppose $\c K$ is a field of characteristic $p$,  
$\c K_{sep}$ is a separable closure of $\c K$ and 
$\c G=\Gal (\c K_{sep}/\c K)$. We assume that $\c G$ acts on $\c K_{sep}$ 
as follows: if  
$g_1,g_2\in\c G$ and $a\in\c K_{sep}$ then 
$g_1(g_2a)=(g_1g_2)a$. Denote by $\sigma $ the 
morphism of taking $p$-th power in $\c K_{sep}$.

In \cite{Ab1, Ab2} we developed a nilpotent analogue of the classical 
Artin-Schreier theory of cyclic field extensions of characteristic $p$. 
We are going to use the covariant analog of this theory, 
cf. the discussion in \cite{Ab10}, 
for explicit description of the group $\c G_{<p}=\c G/\c G^pC_p(\c G)$ as follows.

\subsection{Lie algebra $\c L$} \label{S1.1} 

Suppose $\c K=k((t))$ where $t$ is a fixed uniformizer  
and $k\simeq\F _{p^{N_0}}$ with $N_0\in\N $. 
Fix $\alpha _0\in k$ such that 
$\op{Tr}_{k/\F _p}(\alpha _0)=1$.

Let $\Z ^+(p)=\{a\in\N\ |\ \op{gcd}\,(a,p)=1\}$ and 
$\Z ^0(p)=\Z ^+(p)\cup\{0\}$. 

Let  
$\wt{\c L}$ be a profinite free Lie $\F _p$-algebra with the 
(topological) module of 
generators $\c K^*/\c K^{*p}$ and   
$\c L=\wt{\c L}/C_p(\wt{\c L})$. We can obtain the set   
$$\{D_{0}\}\cup\{D_{an}\ |\ a\in\Z ^+(p), n\in\Z/N_0\}$$ 
of topological 
generators of $\c L_k$
via the following identifications:
$$(\c K^*/\c K^{*p})\hat\otimes _{\F _p}k=
\op{Hom}_{\F _p}(\c K/(\sigma -\id )\c K\,,k)=$$
$$\op{Hom}_{\F _p}(
\oplus _{a\in\Z ^+(p)}kt^{-a}\oplus\F _p\,\alpha _0\,,k)
=\prod_{a\in\Z ^+(p)}\Hom _{\F _p}(kt^{-a},k)\times kD_0$$
and $\Hom _{\F _p}(kt^{-a},k)=\prod _{n\in\Z /N_0}kD_{an}$, 
where for any $\alpha\in k$ and $a,b\in\Z ^+(p)$, 
$D_{an}(\alpha t^{-b})=\delta _{ab}\,\sigma ^n(\alpha )$. 
Note also that the first identification uses 
the Witt pairing \cite{Fo, AJ} and $D_{0}$ comes  from 
$t\otimes 1\in\,(\c K^*/\c K^{*p})\hat\otimes _{\F _p}k$. 

For any $n\in\Z /N_0$, set 
$D_{0n}=t\otimes (\sigma ^n\alpha _0)=(\sigma ^n\alpha _0)D_0$.

\subsection{Groups and Lie algebras of nilpotent class $<p$} \label{S1.2} 

The basic ingredient of the nilpotent Artin-Schreier 
theory is the equivalence of the category of 
$p$-groups of nilpotent class $s_0<p$ and the 
category of Lie $\Z _p$-algebras of the same nilpotent class $s_0$,  
\cite{La, Kh}. 
In the case of objects killed by $p$, this 
equivalence can be explained as follows. 

Let $L$ be a Lie  $\F _p$-algebra of nilpotent class $<p$, i.e. $C_p(L)=0$. 

Let $A$ be an enveloping algebra  of $L$. Then there is a natural embedding 
$L\subset A$, the elements of $L$ generate the augmentation ideal $J$ of $A$ 
and we have a morphism of algebras $\Delta :A\To A\otimes A$ uniquely determined by the 
condition $\Delta (l)=l\otimes 1+1\otimes l$ for all $l\in L$.

Applying the Poincar\'e-Birkhoff-Witt Theorem as in  
\cite{Ab1} Sect.\,1.3.3, we obtain that:

--- $L\cap J^p=0$; 

--- $L\,\op{mod}\,J^p=\{a\,\op{mod}
\,J^p\ |\ \Delta (a)\equiv a\otimes 
1+1\otimes a\,\op{mod}\,(J\otimes 1+1\otimes J)^p\}\, ;$

--- the set $\wt{\exp}(L)\,\op{mod}J^p$ is identified with the set of all 
''diagonal elements modulo degree $p$``, i.e. with 
the set of $a\in 1+J\,\op{mod}\,J^p$ such that 
$\Delta (a)
\equiv a\otimes a\,\op{mod}(J\otimes 1+1\otimes J)^p$.  
(Here  $\wt{\exp}(x)=\sum _{0\leqslant i<p}x^i/i!$ 
is the truncated exponential.) 
\medskip 

In particular,  there is a natural embedding 
$L\subset A/J^p$ and in terms of this embedding 
the Campbell-Hausdorff formula appears as    
$$(l_1,l_2)\mapsto l_1\circ l_2=
l_1+l_2+\frac{1}{2}[l_1,l_2]+\dots ,\ \ l_1,l_2\in L\, ,$$
where $\wt{\exp}(l_1)
\wt{\exp}(l_2)\equiv \wt{\exp}(l_1\circ l_2)\,\op{mod}\,J^p$.   
This composition law provides the set $L$ with 
a group structure and we denote this group by $G(L)$. Note that  a subset 
$I\subset L$ is an ideal in $L$ iff 
$G(I)$ is a normal subgroup in $G(L)$. 
Clearly, $G(L)$  has exponent $p$ and nilpotent class $<p$. 
Then the correspondence 
$L\mapsto G(L)$ is the above mentioned  
equivalence of the categories of $p$-groups of 
exponent $p$ and nilpotent class $s<p$ 
and Lie 
$\F _p$-algebras of the same nilpotent class $s$. 
This  
equivalence can be naturally extended to the categories of 
pro-finite Lie algebras and 
pro-finite $p$-groups.

\subsection{Epimorphism  $\eta _e:\c G\To G(\c L)$} \label{S1.3} 

Let $L$ be a finite Lie $\F _p$-algebra 
of 
nilpotent class $<p$ and set $L_{sep}:=L_{\c K_{sep}}$. The elements of 
$\c G=\Gal (\c K_{sep}/\c K)$ and  
$\sigma $ act on $L_{sep}$ through the second factor, 
$L_{sep}|_{\sigma =\id}=L$ and $(L_{sep})^{\c G}=L_{\c K}$. 
The covariant nilpotent Artin-Schreier theory states that  
for any $e\in G(L_{\c K})$,  the set 
$$\c F(e)=\{f\in G(L_{sep})\ |\ \sigma (f)=e\circ f\}$$  
is not empty and for any fixed $f\in \c F(e)$, the map  
$\tau\mapsto (-f)\circ \tau (f)$ is a continuous group homomorphism 
$\pi _f(e):\c G\To G(L)$.  The correspondence $e\mapsto\pi _f(e)$ has 
the following properties: 
\medskip 

a) if $f'\in\c F(e)$ then $f'=f\circ l$, where $l\in G(L)$, and 
$\pi _f(e)$ and $\pi _{f'}(e)$ are conjugated via $l$; 
\medskip 

b) for any continuous group homomorphism $\pi :\c G\To G(L)$, 
there are $e\in G(L_{\c K})$ and 
$f\in \c F(e)$ such that $\pi _f(e)=\pi $;
\medskip 

c) for appropriate elements $e,e'\in G(L_{\c K})$ and 
$f,f'\in G(L_{sep})$, we have 
$\pi _f(e)=\pi _{f'}(e')$ iff 
there is an $x\in G(L_{\c K})$ such that $f'=x\circ f$ and, therefore,   
$e'=\sigma (x)\circ e\circ (-x)$. 
\medskip 

In 
\cite{Ab1,Ab2,Ab3} we applied this theory 
to the Lie algebra $\c L$ from Sect.\ref{S1.1} 
via a special choice of $e\in\c L_{\c K}$. 
Now we just assume that 

\begin{equation} \label{E1.1}   e\equiv \sum \limits _{a\in\Z ^0(p)} t^{-a}D_{a0}
\,\op{mod}\,C_2(\c L_{\c K})\, .
\end{equation}

Under this assumption the map 
$\pi _f(e)\,\op{mod}\,\c G^pC_2(\c G)$ 
induces a group isomorphism of 
$\c G^{ab}\hat\otimes\F _p$ and $G(\c L)/C_2(G(\c L))
=\c L^{ab}=\c K^*/\c K^{*p}$,  
which coincides with the inverse to the reciprocity map 
of local class field theory, cf.  \cite{AJ}. 
This also implies that $\pi _f(e)$ (when taken modulo 
$\c G^pC_p(\c G)$) induces a group isomorphism 
$\c G_{<p}\simeq G(\c L)$. 
We agree to fix 
a choice of $f\in\c F(e)$ and 
use the notation  $\eta _e=\pi _f(e)$. So, at this stage, 
$\eta _e$ is just an arbitrary lift of the canonical 
isomorphism of local class field theory. 

\subsection{Auxiliary fields $\c K'_{\gamma }$}\label{S1.4} 

Our approach to the ramification filtration in $\c G_{<p}$ substantially 
uses the construction of a totally 
ramified extension $\c K'$ of $\c K$ such that 
$[\c K':\c K]=q$ and 
the Herbrand function $\varphi _{\c K'/\c K}$ 
has only one edge point $(r^*,r^*)$. Here $q=p^{N^*}$ with $N^*\in\N $, and 
$r^*=b^*/(q-1)$, where $b^*\in\Z ^+(p)$.  For simplicity, we assume 
that $N^*\equiv 0\,\op{mod}\,N_0$, i.e. $\sigma ^{N^*}$ acts 
as identity on the residue field $k$ of $\c K$. More substantial restrictions 
on these  parameters will be introduced in 
Sect.\ref{S2.1}.
\medskip

For a detailed explanation of the construction of 
$\c K'$ cf. e.g. \cite{Ab3}, Sect.1.5. We just recall that 
if $r^*=m/n$ with coprime $m,n\in\N $, 
then $\c K'=\c K(U^n)\subset\c K(u)(U)$, where $u^n=t$ and 
$U^q+r^*U=u^{-m}$. 
We can apply Hensel's Lemma 
to choose a uniformizer $t_1$ in $\c K'$ such that 
$t=t_1^qE(t_1^{b^*})^{-1}$, where    
$E(X)=\exp (X+X^p/p+\dots +X^{p^n}/p^n+\dots )
\in\Z _p[[X]]$ is the Artin-Hasse exponential. 
\medskip

We need the following generalization of the construction of $\c K'$.

For $\gamma\in\Z /p\setminus\{0\}$, let the field 
$\c K'_{\gamma }=k((t_{\gamma }))$ be such that:
\medskip  

a)\ $[\c K'_{\gamma }:\c K]=q$;
\medskip 

b)\ $\varphi _{\c K'_{\gamma }/\c K}(x)$ has only 
one edge point $(r^*,r^*)$;
\medskip 

c)\ $\c K'_{\gamma }=k((t_{\gamma }))$, where 
$t=t^{q}_{\gamma }E(\gamma  t^{b^*}_{\gamma  })^{-1}$. 
\medskip 

The fields $\c K'_{\gamma }$ appear in the same 
way as the field $\c K'$.  More precisely,   
$\c K'_{\gamma }=\c K(U_{\gamma }^n)\subset \c K(u)(U_{\gamma })$, where 
$u^n=t$ and $U_{\gamma }^q+\gamma r^*U_{\gamma }=
u^{-m}$. Note that $\c K'_{\gamma }$  
is separable over $\c K$ (but generally is not a $p$-extension over $\c K$).

\subsection{The criterion} \label{S1.5} 

Suppose $\c K'_{\gamma }$ is the 
field from Sect.\ref{S1.4}. 
Consider the field isomorphism $\iota _{\gamma }:
\c K\To \c K'_{\gamma }$ such that 
$\iota _{\gamma }:t\mapsto t_{\gamma }$ and 
$\iota _{\gamma }|_k=\id _k$. 
Let $e_{\gamma }=(\id _{\c L}\otimes\iota _{\gamma })e$.  
Then   $\sigma ^{N^*}e_{\gamma }=e(t_{\gamma  }^q)$ (this is the result of the substitution 
$t\mapsto t_{\gamma }^q$ to $e=e(t)$).

Choose 
$f_{\gamma }\in \c F(e_{\gamma })$ and consider  
$\pi _{f_{\gamma }}(e_{\gamma }):\Gal (\c K_{sep}/\c K'_{\gamma })\To 
G(\c L)$.  

For $Y\in \c L_{sep}$ and an ideal $\c I$ 
in $\c L$,
define the field of definition  
of $Y\,\op{mod}\,\c I_{sep}$ over, say,  $\c K$ 
as 
$$\c K(Y\,\op{mod}\,\c I_{sep}):=\c K_{sep}^{\c H}\, ,$$ 
where $\c H=\{g\in\c G\ |\ (\id _{\c L}\otimes g)Y
\equiv Y\,\op{mod}\,\c I_{sep}\}$.  

For any field extension $\c E'/\c E$ in $\c K_{sep}$,  
define the biggest ramification number 
$$v(\c E'/\c E)=\op{max}\{v\ |\ \Gal (\c K_{sep}/\c E)^{(v)} 
\text{acts non-trivially on } 
\c E'\}\, .$$ 

The methods from \cite{Ab1, Ab2, Ab3} are based 
on the following criterion.  

Suppose $v_0>0$, $r^*<v_0$ and the auxiliary fields $\c K'_{\gamma }$ 
correspond to the parameters $r^*$ and $N^*$ (with $q=p^{N^*}$).

\begin{Prop} \label{P1.1}
 Suppose   
$f=X_{\gamma }\circ \sigma ^{N^*}(f_{\gamma })$. 
Then $\c L^{(v_0)}$ is the minimal ideal in the family of 
all ideals $\c I$ of $\c L$ such that 
$$v(\c K'_{\gamma }(X_{\gamma }\,\op{mod}\,\c I_{sep})/\c K'_{\gamma })<
qv_0-b^*.$$ 
\end{Prop}

The proof goes along the lines 
of the proof for $\gamma =1$, cf.  e.g. 
 \cite{Ab3}, Sect.1.6.  
It is based just on the following 
elementary properties of the upper ramification numbers: 
\medskip 
  
{\it  if $v=v(\c K(f\,\op{mod}\,\c I_{sep})/\c K)$ then: 
\medskip

-- \  $v(\c K_{\gamma }'(f_{\gamma }\,
\op{mod}\,\c I_{sep})/\c K_{\gamma }')=v$; 
\medskip 

-- \  $v(\c K_{\gamma }'(f_{\gamma }\,\op{mod}\,
\c I_{sep})/\c K)=\varphi _{\c K_{\gamma }'/\c K}(v)$; 
\medskip 

-- \  if $v>r^*$ then $\varphi _{\c K_{\gamma }'/\c K}(v)=r^*+(v-r^*)/q<v$. }
\medskip 

Note that  
$f=X_{\gamma } \circ \sigma ^{N^*}f_{\gamma }$ implies that 
$e(t)=\sigma X_{\gamma }\circ \sigma ^{N^*}e_{\gamma }\circ (-X_{\gamma })$. 
\medskip 

Vice versa, suppose $X\in\c L_{sep}$ and 
\begin{equation} \label{E1.2} 
e(t)=\sigma X\circ \sigma ^{N^*}e_{\gamma }\circ (-X)\, .
\end{equation} 
 Then 
$l=(-\sigma ^{N^*}f_{\gamma })
\circ (-X)\circ f\in\c L_{sep}|_{\sigma =\id }=\c L$ 
and replacing $f_{\gamma }$ by 
$f_{\gamma }\circ l\in\c F(e_{\gamma })$ we obtain 
$f=X\circ \sigma ^{N^*}f_{\gamma }$. Therefore, 
in Prop.\ref{P1.1} we can use  
identity \eqref{E1.2} instead of the 
identity $f=X_{\gamma }\circ \sigma ^{N^*}f_{\gamma }$. 

Note that for any $\gamma $, there is a unique field isomorphism 
$\iota '_{\gamma }:\c K'_{\gamma }\To \c K$ such that 
$\iota '_{\gamma }(t_{\gamma })=t$ and 
$\iota '_{\gamma }|_k=\id $. Therefore, if we set 
$e^{(q)}:=e(t^q)$ and 
$\gamma *e^{(q)}:=e(t^qE(\gamma t^{b^*})^{-1})$ then 
Prop.\ref{P1.1} can be stated in the following equivalent form. 

\begin{Prop} \label{P1.2} 
If $X_{\gamma }\in \c L_{sep}$ is such that 
$$\gamma *e^{(q)}=
\sigma X_{\gamma }\circ e^{(q)}\circ (-X_{\gamma })\, $$  
then $\c L^{(v_0)}$ is the minimal ideal in the set of all ideals 
$\c I$ of $\c L$ such that 
$$v(\c K(X_{\gamma }\,\op{mod}\,\c I_{sep})/\c K)< qv_0-b^*\,.$$ 
\end{Prop}

Suppose $\wt{\c J}\subset\c L$ is a closed ideal and 
$\pi :\c L\To L:=\c L/\wt{\c J}$ 
is a natural projection. Then we can use $e_L=\pi _{\c K}(e)\in L_{\c K}$, 
$f_L:=\pi _{sep}(f)\in L_{sep}$, $\eta _{e_L}=\pi\eta _e:\c G\To G(L)$ 
and $X_{\gamma L}:=\pi _{sep}(X_{\gamma })$ to state the following  analog 
of Prop.\ref{P1.2}. 

\begin{Prop} \label{P1.3} 
$L^{(v_0)}:=\eta _{e_L}(\c G^{(v_0)})$ is 
the minimal ideal in the set of all ideals 
$\c I$ of $L$ such that 
$v(\c K(X_{\gamma L}\,\op{mod}\,\c I_{sep})/\c K)< qv_0-b^*$.   
\end{Prop}

\subsection{Lie algebra $\bar{\c L}$ and 
epimorphism $\eta _{\bar e}$}\, \label{S1.6}

Introduce the weight function (i.e.\,valuation) $\op{wt}:\c L_k\To \N $ on $\c L_k$ by  
setting on its generators $\op{wt}(D_{an})=s$ if $(s-1)v_0\leqslant a<sv_0$. 
We obtain a decreasing central 
filtration by the ideals $\c L(s)=\{l\in\c L\ |\ 
\op{wt}(l)\geqslant s\}$ of $\c L$ such that $\c L(1)=\c L$.  This weight function 
gives us also a decreasing filtration of ideals 
$\c J(s)$ in the enveloping algebra $\c A$ such that 
$\c J(1)=\c J$ and for any $s$, 
$(\c J(s)+\c J^p)\cap \c L=\c L(s)$ 
(use the Poincar\'e-Birkhoff-Witt theorem).

Consider the $k$-submodule ${\c N}$ in ${\c L}_{\c K}$ generated by 
all $t^{-b}l$, where for some $s\geqslant 1$, 
$l\in {\c L}(s)_k$ and $b<sv_0$. 
Then ${\c N}$ has a natural structure of a 
Lie algebra over $k$. 
For any $i\geqslant 0$, 
let ${\c N}(i)$ be the $k$-submodule 
in ${\c L}_{\c K}$ generated by all $t^{-b}l$ 
where $l\in{\c L}(s)$ and 
$b<(s-i)v_0$. Then ${\c N}(i)$ is an ideal in 
${\c N}$.

Let 
$\bar{\op{pr}}:\c L\To\bar{\c L}:=\c L/\c L(p)$ be a natural projection. 
Then $\bar{\c L}(s)=\bar{\op{pr}}(\c L(s))$ 
is a decreasing central 
filtration in $\bar{\c L}$ such that $\bar{\c L}(p)=0$. 
Let $\bar{\c N}\subset \bar{\c L}_{\c K}$ be an analog of $\c N$ 
(where the algebra $\bar{\c L}$ is used instead of $\c L$). 
 
For $i\geqslant 0$, let $\bar{\c N}(i)$ be the appropriate 
ideals in $\bar{\c N}$. 
Note that $\bar{\c N}(p-1)\subset\bar{\c L}_{\m }$, where 
$\m =tk[[t]]$ (use that $\bar{\c L}(p)=0$), and 
introduce the Lie 
algebra $\wt{\c N}=\bar{\c N}/\bar{\c N}(p-1)$. 

We assume (in addition to \eqref{E1.1}) that:
\begin{equation} \label{E1.3}  e\in\c N\, 
\end{equation}
(now $\eta _e$ is not an arbitrary lift of the 
reciprocity map of class field theory but it is still 
quite general).

Let $\bar e:=\bar{\op{pr}}_{\c K}e\in\bar{\c N}$ and 
$\bar f:=(\bar{\op{pr}}_{sep})f\in\bar{\c L}_{sep}$. 
If $\eta _{\bar e}:=\bar{\op{pr}}\cdot \eta _e$ 
then for any $\tau\in\c G$, 
$\bar{\eta }_{\bar e}(\tau )=(-\bar f)\circ \tau\bar f$. 
Verify that $\eta _{\bar e}$ depends only on  
$\wt e:=\bar e\,\op{mod}\,\bar{\c N}(p-1)\in\wt{\c N}$. 

\begin{Prop} \label{P1.4} 
Let $\bar e'\in\bar{\c L}_{\c K}$ and $\bar e'\equiv 
\bar e\,\op{mod}\,\bar{\c N}(p-1)$. 
Then there is a unique $\bar f'\in\bar{\c L}_{sep}$ 
such that $\sigma \bar f'=\bar e'\circ \bar f'$ and 
$\bar f'\circ (-\bar f)\in\bar{\c N}(p-1)$.
\end{Prop}

\begin{proof} Note that 
$\sigma $ is topologically nilpotent on 
$\bar{\c N}(p-1)\subset\bar{\c L}_{\m}\,$.  
Prove the existence of  
$\bar x\in\bar{\c N}(p-1)\,$ such that 
$\bar e '= (\bar\sigma x)\circ \bar e\circ (-\bar x)$ 
by induction on $s\geqslant 1$ modulo 
the ideals $\bar{\c L}(s)_{\c K}$ as follows: 

-- if $s=1$ there is nothing to prove; 

-- if $s\geqslant 1$ and $\bar x_s\in\bar{\c N}(p-1)$ is such that 
$\bar e'=(\sigma \bar x_s)\circ \bar e\circ (-\bar x_s)+A_s$ 
with $A_s\in\bar{\c L}(s)_{\c K}$, 
then $A_s\in\bar{\c N}(p-1)\cap\bar{\c L}(s)_{\c K}$. If 
$\delta =-\sum \limits _{m\geqslant 0}\sigma ^m(A_s)$ then 
$\bar x_{s+1}:=\bar x_s+\delta \in \bar{\c N}(p-1)\cap \bar{\c L}(s)_{\c K}$, 
$\sigma \delta -\delta =A_s$ and 
$$\bar e'\equiv (\sigma \bar x_{s+1})\circ 
\bar e\circ (-\bar x_{s+1})\,\op{mod}\,\bar{\c L}(s+1)_{\c K}\,.$$
Clearly, $\bar x=\bar x_p$. 

Now $\bar f'=\bar x\circ \bar f\in\bar{\c L}_{sep}$  
satisfies the 
requirements of proposition.  If $\bar f''\in\bar{\c L}_{sep}$ also has 
such properties then $\bar f''\circ (-\bar f')
\in\bar{\c N}(p-1)\cap \bar{\c L}=0$ and 
$\bar f''=\bar f'$. 
 \end{proof}

\section{Lie algebra $\bar{\c L}^{\dag }$ and ideal   
$\bar{\c L}[v_0]\subset\bar{\c L}$} \label{S2} 

In this section we 
introduce the Lie $\F _p$-algebra $\bar{\c L}^{\dag }$ together with the  
epimorphism of Lie algebras ${\c V}:
\bar{\c L}^{\dag }\To \bar{\c L}$ and its section 
$(j^0)^{-1}:\bar{\c L}\simeq \bar{\c L}^{\dag }[0]\subset\bar{\c L}^{\dag }$. Let  
$\alpha _{p}=\Spec\, \F _p[U]$, $U^p=0$,  
be the formal group scheme over $\F _p$ with the 
coaddition $\Delta (U)=U\otimes 1+1\otimes U$. We introduce the coaction 
$\Omega _U:\bar{\c L}^{\dag }\To \F _p[U]\otimes \bar{\c L}^{\dag }$ 
of $\alpha _{p}$ on $\bar{\c L}^{\dag }$ and use it to define and 
characterize the ideal $\bar{\c L}[v_0]$ of $\bar{\c L}$.

 \subsection{Parameters $r^*$ and $N^*$} \label{S2.1} \ 
Fix $u ^*\in\N $ and $w^*>0$. (Below we will specify $u^*=(p-1)(p-2)+1$ 
and $w^*=(p-1)v_0$.) 
\medskip 

For $1\leqslant s<p$, denote 
by $\delta _0(s)$ the minimum of (strictly) positive values of 
$$v_0-\frac{1}{s}(a_1+a_2/p^{n_2}+\dots +a_u/p^{n_u})\, ,$$ 
where $u\leqslant u^*$,   
all $n_i\in\Z _{\geqslant 0}$ and $a_i\in [0,w^*)\cap\Z $. 
The existence of such $\delta _0(s)$ can be proved easily by induction on $u$ 
for any fixed $s$.  
\medskip 

Set $\delta _0:=\min\{\delta _0(s)\ |\ 1\leqslant s<p\}$. 
\medskip 

Let $r^*\in\Q $ be such that 
$r^*=b_0^*/(q^*_0-1)$, where $q^*_0=p^{N^*_0}$ with 
$N^*_0\geqslant 2$, $b^*_0\in\N $  
and $\op{gcd}(b_0^*,\, p(q^*_0-1))=1$. The set of 
such $r^*$ is dense 
in $\mathbb R _{>0}$ 
and we can assume that $r^*\in (v_0-\delta _0, v_0)$. 
\medskip

For $1\leqslant u\leqslant u^*$, introduce the following subsets in $\Q $\,:
\medskip   

--- $A[u]$ is the set of all 
$$a_1p^{-n_1}+a_2p^{-n_2}+\dots +a_up^{-n_u}\, ,$$ 
where 
$0=n_1\leqslant \dots \leqslant n_u$, 
 all   
$a_i\in [0,w^*)\cap\Z $.  
If $M\in\Z _{\geqslant 0}$ 
we denote by $A[u,M]$ the subset 
of $A[u]$ consisted of the elements satisfying the additional 
restriction $n_u\leqslant M$. Note that $A[u,M]$ is finite. 
\medskip 

--- $B[u]$ is the set of all numbers 
$$r^*(b_1p^{-m_1}+b_2p^{-m_2}+\dots +b_up^{-m_u})\, ,$$ 
where all $0=m_1\leqslant \dots \leqslant m_u$, 
$b_i\in\Z _{\geqslant 0}$, $b_1\ne 0$ and $b_1+\dots +b_u<p$. 
(In particular, $0\notin B[u]$.) 
For $M\in\Z _{\geqslant 0}$, $B[u,M]$ is the subset of $B[u]$ 
consisted of the elements satisfying the additional 
restrictions $m_u\leqslant M$. The set $B[u,M]$ 
is also finite. 
\medskip 

\begin{Lem} \label{L2.1}
 For any $u$, $A[u]\cap B[u]=\emptyset $. 
\end{Lem}

\begin{proof}
 Note that $A[u]\subset\Z [1/p]$.
Prove that $B[u]\cap\Z [1/p]=\emptyset $. 

It will be sufficient to verify that 
for any $n_1,\dots ,n_u, b_1,\dots, b_u\in\Z _{\geqslant 0}$ 
such that 
$0<b_1+\dots +b_u<p$, we have  
$$b_1p^{n_1}+\dots +b_up^{n_u}\not\equiv 0\,\op{mod}\,(q^*_0-1)\, .$$ 
Since $q^*_0\equiv 1\,\op{mod}\,(q^*_0-1)$ we can assume that all $n_i<N_0^*$. 
But then $0<b_1p^{n_1}+\dots +b_up^{n_u}\leqslant (p-1)p^{N_0^*-1}<q^*_0-1$. 
The lemma is proved. 
\end{proof}

For $\alpha, \beta\in\Q $, set $\rho (\alpha,\beta )=|\alpha -\beta |$. 

\begin{Lem}   \label{L2.2}
 If $\alpha\notin B[u]$ then 
$$\rho (\alpha ,B[u]):=\inf\{\rho (\alpha ,\beta )\ |\ 
\beta\in B[u]\}\ne 0\, .$$
\end{Lem}
\begin{proof} 
 Use induction on $u$. 

If $u=1$ there is nothing to prove because $B[1]$ is finite. 

Suppose $u\geqslant 1$ and $\rho (\alpha ,B[u])>0$. 

Choose $M_u\in\Z _{\geqslant 0}$ such that 
$r^*(p-1)/p^{M_u+1}<\rho (\alpha ,B[u])/2$. 

If $\beta\in B[u+1]\setminus B[u+1,M_u]$ then 
there is $\beta '\in B[u]$ such that 
$\rho (\beta ,\beta ')<\rho (\alpha ,B[u])/2$. Then 
$$\rho (\alpha ,\beta )\geqslant 
\rho (\alpha ,\beta ')-\rho (\beta ',\beta )
\geqslant \rho (\alpha ,B[u])-\rho (\alpha ,B[u])/2=\rho (\alpha ,B[u])/2\,,$$ 
and we obtain  
$$\rho (\alpha ,B[u+1])\geqslant 
\min\{\rho (\alpha ,B[u+1,M_u]), \rho (\alpha ,B[u])/2\}>0\, .$$
The lemma is proved. 
\end{proof}

\begin{Lem} \label{L2.3}
 If $\beta\notin A[u]$ then $\rho (\beta ,A[u])\ne 0$. 
\end{Lem}

\begin{proof} 
 The proof is similar to the proof of above Lemma \ref{L2.2}. 
\end{proof}

\begin{Lem} \label{L2.4} 
For all $u_1,u_2\leqslant u^*$,  $\rho (A[u_1], B[u_2])>0$. 
\end{Lem}

\begin{proof}
If $u_1=1$ this follows from Lemma \ref{L2.2} because $A[1]$ is finite. 

Suppose $u_1\geqslant 1$ and $\rho (A[u_1],B[u_2])=\delta >0$. 

Choose $M_1\in\Z _{\geqslant 0}$ such that $w^*/p^{M_1}<\delta /2$. 

If $\alpha\in A[u_1+1]\setminus A[u_1+1,M_1]$ 
then there is $\alpha '\in A[u_1]$ such that 
$\rho (\alpha ,\alpha ')<\delta /2$. 
Then for any $\beta\in B[u_2]$, we have 
$$\rho (\alpha ,\beta )\geqslant \rho (\alpha ',\beta )-
\rho (\alpha ,\alpha ')>\delta /2\,.$$ 
Therefore, for any $\alpha\in A[u_1+1]$, 
$$\rho (\alpha ,B[u_2])\geqslant \min 
\{\,\rho (A[u_1+1,M_1], B[u_2]),\ \delta /2\,\}\}>0 \,.$$ 
Lemma is proved. 
\end{proof}
\medskip 
 
{\sc Fix the values $u^*=(p-1)(p-2)+1$ and $w^*=(p-1)v_0$} 
\newline 
(since $u^*\geqslant p-1$, $B[u^*]=B[p-1]$). 
\medskip 

Choose $N^*\in\N $ satisfying the 
following conditions:
\medskip 

{\bf C1)}\  $N^*\equiv 0\,\op{mod}\,N^*_0$;
\medskip

{\bf C2)}\  $p^{N^*}\rho (A[u^*], B[u^*])\geqslant 2r^*(p-1)$;  
\medskip 

{\bf C3)}\  $r^*(1-p^{-N^*})\in (v_0-\delta _0, v_0)$. 
\medskip

Introduce $q=p^{N^*}$ and $b^*=b^*_0(q-1)/(q_0-1)\in\N $.
\medskip  

Note that $r^*=b^*/(q-1)$ and $b^*\in\Z ^+(p)$. 
\medskip

\begin{Prop} \label{P2.5} 
 If $\alpha\in A[u^*]$ and $\beta\in B[u^*]$ 
then  
$$q\,|q\alpha -(q-1)\beta |>b^*(p-1)\, .$$ 
 \end{Prop}

\begin{proof} 
Indeed, the left-hand side of our inequality equals 
$$q\,|q\alpha -(q-1)\beta |=q^2|\alpha -\beta +\beta /q|
\geqslant q^2|\alpha -\beta |-
\beta q
\geqslant q^2\rho (A[u^*],\, B[u^*])$$
$$-r^*(p-1)q
\geqslant 2r^*(p-1)q-r^*(p-1)q=r^*(p-1)q>b^*(p-1)\, .$$
\end{proof}

 \subsection{The set $\mathfrak{A}^0$} \label{S2.2} 
 Use the above parameters $r^*$, $N^*$, $q=p^{N^*}$.   

\begin{definition}  ${\mathfrak{A}}^0$ is the set of all  
$\iota =p^m(q\alpha -(q-1)\beta )$,  where   $m\in\Z _{\geqslant 0}$,   
$\alpha \in A[u^*,m]$, $\beta \in B[u^*,m]\cup\{0\}$ and $|\iota |\leqslant  b^*(p-1)$. 
(Note that $p^m\alpha \in\Z _{\geqslant 0}$ and $p^m\beta /r^*\in\N $.)
\end{definition}

 Let $\mathfrak{A}^0_0:=\{\iota\in\mathfrak {A}^0\ |\ \beta =0\}$.

\begin{Lem} \label{L2.6} 
Suppose $\iota =p^m(q\alpha -(q-1)\beta )\in\mathfrak{A}^0$. Then:
\medskip 

{\rm a)} $\mathfrak{A}^0_0=\{qa\ |\ a\in [0,(p-1)v_0)\cap\Z \,\}$; 
\medskip 

{\rm b}\ if $\beta\ne 0$ then $m<N^*$ 
(in particular, $\mathfrak{A}^0$ is finite);
\medskip 

{\rm c)} the integers $p^m\alpha $ and $p^m\beta /r^*$ 
do not depend on the presentation of 
$\iota $ in the form $p^m(q\alpha -(q-1)\beta )$ 
from the definition of $\mathfrak{A}^0$. 
\end{Lem} 

\begin{proof} a) If $\iota\in\frak{A}^0_0$ then $\iota =qp^m\alpha\in\frak{A}^0_0\subset\frak{A}^0$ 
means that 
$p^m\alpha /(p-1)\leqslant b^*/q=r^*(1-q^{-1})\in (v_0-\delta _0,v_0)$. 
By the choice of $\delta _0$ from  Sect.\ref{S2.1}, the inequalities 
$p^m\alpha /(p-1)<v_0$ and $p^m\alpha /(p-1)\leqslant v_0-\delta _0$ are equivalent. 
Therefore, $\frak{A}^0_0\subset\{qa\ |\ a\in [0,(p-1)v_0)\cap\Z\}$. 
The opposite embedding is obvious. 
\medskip

b) If $\beta\in B[u^*,m]$ and $m\geqslant N^*$ then by Prop.\ref{P2.5}, 
 $|\iota |> b^*(p-1)$ i.e. 
$\iota \notin \mathfrak{A}^0$. 
\medskip

c)  
If $\iota =p^{m'}(q\alpha '-(q-1)\beta ')$ 
is another presentation 
of $\iota $ then 
$p^m\beta /r^*$ and $p^{m'}\beta '/r^*$ are 
non-negative congruent modulo $q$ integers  
and the both are smaller than $q$. Indeed, 
if $\beta /r^*=b_1+b_2p^{-m_2}+\dots +b_up^{-m_u}$, 
where all $0\leqslant m_i\leqslant m$ and $u\leqslant u^*$, then 
$$p^m\beta /r^*\leqslant p^m(b_1+\dots +b_u)
\leqslant p^m(p-1)<p^{m+1}\leqslant q$$
because $m<N^*$. Similarly, $p^{m'}\beta '/r^*<q$.  
Therefore, 
they coincide and this implies 
also that 
$p^m\alpha =p^{m'}\alpha '$. 
 \end{proof}
 
 \begin{Cor} \label{C2.7}
  Suppose that  $\iota =p^m(q\alpha -(q-1)\beta )\in\mathfrak{A}^0$.  
  Then the sum of the \lq\lq\,$p$-digits\rq\rq\  
  $b_1+\dots +b_u$ of the appropriate $\beta /r^*
  =b_1+b_2p^{-m_2}+\dots +b_up^{-m_u}$ depends only on $\iota $. 
 \end{Cor}
 
 \begin{definition} 
  $\op{ch}(\iota ):=b_1+\dots +b_u$\,. 
 \end{definition} 
 
In the notation from Sect.\ref{S2.2} suppose 
$\iota =p^m(q\alpha -(q-1)\beta )\in{\mathfrak{A}}^0$.   
By Lemma \ref{L2.6} $p^m\alpha $ depends only on $\iota $ and 
can be presented (non-uniquely) in the form 
$a_1p^{n_1}+a_2p^{n_2}+\dots +a_up^{n_u}$ where all coefficients 
$a_i\in [0,(p-1)v_0))\cap \Z $, 
$0\leqslant n_i\leqslant m$, $n_1=m$ and $u\leqslant u^*$. 

\begin{definition} 
 $\kappa (\iota )$ is the maximal natural number 
such that for any above presentation  
 of $p^m\alpha $, $\kappa (\iota )\leqslant u$. 
\end{definition}

\begin{remark}  
a) If $\iota\in\mathfrak{A}^0$ then 
$\kappa (\iota )\leqslant u^*$ and $\op{ch}(\iota )\leqslant p-1$;

b) if $\iota\in\mathfrak{A}^0_0$ then $\op{ch}(\iota )=0$;

c) if $\iota\in\mathfrak{A}^0_0$ and $\iota\ne 0$ then $\kappa (\iota )=1$.
\end{remark}
\medskip 

\subsection{Lie algebras ${\c L}^{\,\dag }$ 
and $\bar{\c L}^{\,\dag }$} \label{S2.3}

Suppose 
$\iota =p^m(q\alpha -(q-1)\beta )\in\mathfrak{A}^0$ 
is given in the standard notation from Sect.\ref{S2.2}. 
Let $w^0(\iota )$ be the minimal (positive) natural number such that 
$\iota < w^0(\iota )b^*$.  

\begin{definition} The subset 
$\mathfrak{A}^+(p)$ consists of  $\iota\in\mathfrak{A}^0$ such that 
 
--- $\iota >0$;

--- $\op{gcd}(p^m\alpha \,,p^m\beta /r^*,\,p)=1$;

--- $w^0(\iota )+\op{ch}(\iota )\leqslant p-1$;

--- $\kappa (\iota )\leqslant (p-2)\op{ch}(\iota )+w^0(\iota )$. 
\end{definition} 

\begin{remark}
For any $\iota\in\mathfrak{A}^+(p)$, 
$(p-2)\op{ch}(\iota )+w^0(\iota )\leqslant
(p-2)^2+p-1=u^*\, .$
\end{remark}

The elements of $\{t^{-\iota }\ |\ \iota\in \mathfrak{A}^+(p)\}$ behave 
\lq\lq well\rq\rq\ modulo 
$(\sigma -\id )\c K$, i.e. the natural map 
$\sum\limits _{\iota\in\mathfrak{A}^+(p)}kt^{-\iota }\To \c K/(\sigma -\id )\c K$ 
is injective. This is implied by the following proposition. 

\begin{Prop} \label{P2.8} 
Let $v_p$ be the $p$-adic valuation such that $v_p(p)=1$. 
\medskip 

{\rm a)} 
Then all $\iota p^{-v_p(\iota )}$, where $\iota\in\mathfrak{A}^+(p)$, 
 are pairwise different. 
\medskip 

{\rm b)} If $\iota\in\mathfrak{A}^+(p)$ and $\op{ch}(\iota )=1$ then 
$\iota p^{-v_p(\iota )}\geqslant qv_0-b^*$.
\end{Prop}

\begin{proof} a) Suppose $\iota =p^m(q\alpha -(q-1)\beta )\in \mathfrak{A}^+(p)$. 
\medskip 

If $\op{ch}(\iota )=0$ then  
$\iota\mapsto \iota p^{-v_p(\iota )}$ identifies $\{\iota\in\mathfrak{A}^+(p)\ |\ 
\op{ch}(\iota )=0\}$ 
with  
$\Z ^+(p)\cap [0,(p-1)v_0)$, cf. Lemma \ref{L2.6}a). 

\begin{remark}
For similar reasons, if $1\leqslant s<p$ 
and $a\in \Z ^+(p)\cap [0,(p-1)v_0)$ then $a<sv_0$ iff $qa<sb^*$.
\end{remark}

If $\op{ch}(\iota )\geqslant 1$ then 
$\iota p^{-m}\notin p\N $, i.e. $m\geqslant v_p(\iota )$. 

Indeed, 
$\iota p^{-m}=q\alpha -(q-1)\beta \in p\N $ implies (use that 
$q\alpha\in p\N $ because $m<N^*$) that 
$p^{-m}(b_1+b_2p^{m_2}+\dots +b_up^{m_u})\in p\N $ where 
all $m_i\in [0,m]$. But this number is $\leqslant \op{ch}\,\iota <p$. 
The contradiction. 

Then by Prop.\ref{P2.5}, 
$\iota p^{-v_p(\iota )}\geqslant \iota p^{-m}=|q\alpha -(q-1)\beta |>$ 
\linebreak 
$(b^*/q)(p-1)=r^*(1-q^{-1})(p-1)> 
(v_0-\delta _0)(p-1)$  
(use property {\bf C3} from Sect.\ref{S2.1}). 

Finally, 
if $\iota \in\mathfrak{A}^0_0$ then $\iota p^{-v_p(\iota )}=a<
(p-1)v_0$ implies that  $a<(v_0-\delta _0)(p-1)$ 
by the choice of $\delta _0$, cf. Sect.\ref{S2.1}. On the other hand, 
for all $\iota\in \mathfrak{A}^+(p)$ with $\op{ch}(\iota )\geqslant 1$, the values  
$\iota p^{-v_p(\iota )}$ are different (use that 
$\op{gcd}(p^m\alpha ,p^m\beta /r^*)\not\equiv 0\,\op{mod}\,p$) 
and bigger than $(v_0-\delta _0)(p-1)$. 
\medskip 

b) Here $\iota p^{-\iota _p(\iota )}=\iota p^{-m}=q\alpha -b^*$. 
If $\alpha\geqslant v_0$ then  
$\iota p^{-v_p(\iota )}\geqslant qv_0-b^*$. 
If $\alpha <v_0$ then 
$\iota p^{-v_p(\iota )}\leqslant q(v_0-\delta _0-r^*(q-1)/q))<0$, cf. condition 
C3) from Sect.\ref{S2.1}. The contradiction. 

The proposition is completely proved.  
\end{proof}

\begin{definition} 
  $\mathfrak{A}^0(p)=\mathfrak{A}^+(p)\cup\{0\}$.
\end{definition}

Let $\wt{\c L}_k^{\,\dag }$ be the Lie algebra over $k$ with the set of 
free generators 
$$\{\ D^{\dag }_{\iota n}\ |\ \iota\in\mathfrak{A}^+(p), 
n\in\Z /N_0\}\cup \{D^{\,\dag }_0\}\, .$$ 
Set (compare with Sect.\ref{S1.1}) 
$D^{\,\dag }_{0n}=\sigma ^n(\alpha _0)D_0^{\,\dag }$, 
use the notation $\sigma $ for 
the $\sigma $-linear automorphism of 
$\wt{\c L}^{\,\dag }_k$ such that  
$\sigma :D^{\,\dag }_{\iota n}\mapsto D^{\,\dag }_{\iota ,n+1}$, and  
introduce the Lie $\F _p$-algebras   
$\wt{\c L}^{\,\dag }:=\wt{\c L}^{\,\dag }_k|_{\sigma =\id }$ and  
$\c L^{\,\dag }=\wt{\c L}^{\,\dag }/C_p(\wt{\c L}^{\,\dag })$.  
Note that $\wt{\c L}^{\dag }\otimes k=\wt{\c L}^{\dag }_k$, 
and this matches the agreement about extensions of scalars from the end of Introduction.  
\medskip

Introduce the $w^0$-weights,  
$w^{0}(D^{\dag }_{\iota n}):=w ^{0}(\iota )$.

Denote by $\{\c L^{\,\dag}(s)\}_{s\geqslant 1}$ the minimal 
central filtration of $\c L^{\,\dag }$ such that 
all $D^{\,\dag }_{\iota n}$ with 
$w^{0}(D^{\,\dag }_{\iota n})\geqslant s$ 
belong to $\c L^{\,\dag }(s)_k$. 
This means that $\c L^{\dag }(s)_k$ is an ideal in $\c L^{\dag }_k$ 
generated as $k$-module by all  
$[\dots [D^{\dag }_{\iota _1n_1}, D^{\dag }_{\iota _2n_2}],
\dots ,D^{\dag }_{\iota _rn_r}]$  
such that $w^{0}(\iota _1)+\dots +w^{0}(\iota _r)\geqslant s$. 
Note that $C_s(\c L^{\dag })\subset \c L^{\dag }(s)$. 
\medskip

Let $\c A^{\dag }$ be the enveloping algebra for $\c L^{\dag }$. 

For $m\in\Z _{\geqslant 0}$, let  
$\c A^{\dag }[m]_k$ be the $k$-submodule in $\c A^{\dag }_k$ generated 
by all monomials $D^{\dag }_{\iota _1n_1}
\dots D^{\dag }_{\iota _rn_r}$ such that 
$\op{ch}(\iota _1)+\dots +\op{ch}(\iota _r)=m$. 

By setting 
$\c A^{\dag }[m]=\c A^{\dag }\cap\c A^{\dag }[m]_k$  
we obtain a   
grading  in the category of $\F _p$-algebras 
${\c A}^{\dag }=\oplus _{m\geqslant 0}{\c A}^{\dag }[m]$ 
and the induced grading $\c L^{\dag }=
\oplus _{m\geqslant 0}\c L^{\dag }[m]$ in the category of 
Lie algebras. 
\medskip

For $s\geqslant 1$, set 
$\c L^{\dag }(s)[m]=\c L^{\dag }(s)\cap\c L^{\dag }[m]$. Then 
$\c L^{\dag }(s)=\oplus_{m\geqslant 0}\c L^{\dag }(s)[m]$. 

Let $\bar{\c L}^{\dag }$ be the quotient of $\c L^{\dag }$ by the ideal 
$\sum\limits _{s+m\geqslant p}\c L^{\dag }(s)[m]$. We have the induced 
central filtration 
$\{\bar{\c L}^{\dag }(s)\}_{s\geqslant 1}$ in $\bar{\c L}^{\dag }$ 
such that $\bar{\c L}^{\dag }(p)=0$. 

We also have the induced gradings 
$\bar{\c L}^{\dag }=\oplus _{m\geqslant 0}\bar{\c L}^{\dag }[m]$ 
and $\bar{\c L}^{\dag }(s)=\oplus _{m\geqslant 0}\bar{\c L}^{\dag }(s)[m]$, 
where 
$\bar{\c L}^{\dag }(s)[m]:=\bar{\c L}^{\dag }(s)\cap \bar{\c L}^{\dag }[m]$. 

\begin{definition}
 If $l\in\c L^{\dag }[m]_k$ or 
$l\in\bar{\c L}^{\dag }[m]_k$, $l\ne 0$,  we set $\op{ch}(l)=m$.
\end{definition}

Clearly, for any $m_1,m_2$, 
$[\bar{\c L}^{\dag }[m_1],\bar{\c L}^{\dag }[m_2]]
\subset \bar{\c L}^{\dag }[m_1+m_2]$.

\subsection{Lie algebra $\wt{\c N}^{sp}$} \label{S2.4}

Let ${\c N}^{\,\dag }$ be the $k$-submodule in ${\c L}^{\dag }_{\c K}$ 
generated by the elements of the form $t^{-b}l$, where 
$l\in{\c L}^{\,\dag }(s)[m]_k$ and $b<(s+m)b^*$; 
${\c N}^{\,\dag }$ is a Lie $k$-subalgebra 
in ${\c L}^{\,\dag }_{\c K}$ and for $j\geqslant 0$,  
$t^{jb^*}{\c N}^{\,\dag }$ are 
ideals in ${\c N}^{\,\dag }$. 

Introduce similarly the $k$-subalgebra 
$\bar{\c N}^{\dag }$ in $\bar{\c L}^{\dag }_{\c K}$ 
(generated by all $t^{-b}l$, where 
$l\in\bar{\c L}^{\,\dag }[m](s)_k$ and $b<(s+m)b^*$) and its ideals 
$t^{jb^*}\bar{\c N}^{\,\dag }$.  
Note that $t^{(p-1)b^*}\bar{\c N}^{\,\dag }\subset \bar{\c L}^{\,\dag }_{\m }$ 
(use that $\bar{\c L}^{\,\dag }(p)=0$).  
\medskip 

Set  $\wt {\c N}^{\,\dag }=\bar{\c N}^{\,\dag }/t^{(p-1)b^*}\bar{\c N}^{\,\dag }$. 
\medskip 

The grading from Sect.\ref{S2.3} induces the gradings 
$\bar{\c N}^{\,\dag }=\oplus _{m\geqslant 0}{\bar{\c N}}^{\,\dag }[m]$ and   
$\wt {\c N}^{\,\dag }=\oplus _{m\geqslant 0}\wt {\c N}^{\,\dag }[m]$.

\begin{definition}  
Let 
$\wt{\c N}^{sp}$ be the $k$-submodule 
in $\wt{\c N}^{\dag }$ generated by the elements of 
the form $t^{-\iota }l$ 
with $l\in\bar{\c L}^{\dag }(s)[m]_k$ such that:

a) $\iota\in\mathfrak{A}^0$; 

b) $\iota +\op{ch}(\iota )b^*<(s+m)b^*$; 

c) $\op{ch}(\iota )\geqslant m$, $\kappa (\iota )\leqslant (p-2)m+s$.
\end{definition}

\begin{remark}
 Condition b) means that $t^{-\iota }l
\in t^{\,\op{ch}(\iota )b^*}\wt{\c N}^{\,\dag }$. 
\end{remark}

It will be convenient to introduce the following modules.  

Let $\c N^{sp}$ be a $k$-submodule in $\c N^{\dag }
\subset \c L^{\dag }_{\c K}$ generated by the elements 
$t^{-\iota }l$ such that $l\in\c L^{\dag }(s)[m]$, 
$\iota\in \mathfrak{A}^0$, $\iota +\op{ch}(\iota )b^*<(s+m)b^*$, 
$\op{ch}(\iota )\geqslant m$ and $\kappa (\iota )\leqslant (p-2)m+s$. 
Then the image of $\c N^{sp}$ under the natural map  
$\c N^{\dag }\To \bar{\c N}^{\dag }\To \wt{\c N}^{\dag }$ 
coincides with $\wt{\c N}^{sp}$. 

Let $\c I^{sp}$ be the submodule in $\c N^{\dag }$ generated by  
$t^{-b}l$ such that $l\in\c L^{\dag }(s)[m]$ and either $s+m\geqslant p$ or 
$b<(s+m)b^*-(p-1)b^*$. Then 
the image of ${\c I}^{sp}$ under the above natural map 
${\c N}^{sp}\To\wt{\c N}^{\dag }$ 
is 0.

\begin{Lem} \label{L2.9} 
{\rm a)}\ 
$\wt{\c N}^{ sp}$ is a Lie subalgebra in $\wt{\c N}^{\,\dag}$. 
\medskip 

{\rm b)}\ For any $j\geqslant 0$, $t^{\,jb^*}\wt{\c N}^{ sp}$ 
is an ideal in $\wt{\c N}^{ sp}$. 
\end{Lem} 
\begin{proof} a) Suppose $w_1=t^{-\iota _1}l_1$ and 
$w_2=t^{-\iota _2}l_2$ belong to ${\c N}^{sp}$. 
Assume that for $j=1,2$, 
$l_j\in {\c L}^{\dag }(s_j)[m_j]$, where  
$s_j=w^{0}(l_j)$ and 
$m_j=\op{ch}(l_j)$.   
We must prove that the image $\tilde w$ of 
$w=[w_1,w_2]\in\c N^{\dag }$ in $\wt{\c N}^{\dag }$ belongs to $\wt{\c N}^{sp}$. 
\medskip

Let $s=s_1+s_2$ and $m=m_1+m_2$. Then $l=[l_1,l_2]\in\c L^{\dag }(s)[m]_k$. 
We can assume that $s+m<p$ (otherwise, 
$l\in\c I^{sp}$ and $\tilde w=0$). 
\medskip 

Verify that $\iota =\iota _1+\iota _2\in\mathfrak{A}^0$.  
We can assume $\iota \geqslant (s+m)b^*-(p-1)b^*$ 
(otherwise, $w\in\c I^{sp}$ and $\tilde w=0$). 
This implies $\iota >-(p-1)b^*$. 

Since $w_1,w_2\in\c N^{sp}$ we have also that 
 $$\iota +(\op{ch}(\iota _1)+\op{ch}(\iota _2))b^*<(s+m)b^*
\leqslant (p-1)b^* \, ,$$
and this implies $\iota <(p-1)b^*$.

We can assume that  
$m'=\op{ch}(\iota _1)+\op{ch}(\iota _2)<p$. 
(Otherwise, for $j=1,2$  
$w_j\in t^{\op{ch}(\iota _j)b^*}{\c N}^{\dag }$,  
$w\in t^{m'b^*}\wt{\c N}^{\dag }\subset\c I^{sp}$ and $\tilde w=0$.) 
In addition, $\kappa (\iota )\leqslant \kappa (\iota _1)+\kappa (\iota _2)
\leqslant (p-2)m+s<u^*\, .$ 
As a result, $\iota \in \mathfrak{A}^0$. 

Finally, 
$\op{ch}(\iota )=m'\geqslant m_1+m_2=m$,  
$[w_1,w_2]\in {\c N}^{sp}$ and $w\in\wt{\c N}^{sp}$.  
\medskip 

b) Suppose $t^{-\iota }l\in\c N^{sp}$ is given in 
terms of the above definition of $\c N^{sp}$. 
We can assume that $s+m<p$ and $\iota \geqslant (s+m)b^*-(p-1)b^*$. 

Prove that the image $\tilde w$ of $w=t^{-\iota +jb^*}l$ 
in $\wt{\c N}^{\dag }$ belongs to 
$\wt{\c N}^{sp}$.  

Let $\iota '=\iota -jb^*$. We can assume that 
$\iota '\geqslant (s+m)b^*-(p-1)b^*$ 
(otherwise, $w\in\c I^{sp}$ and $\tilde w=0$). 
Then $-(p-1)b^*<\iota '\leqslant \iota <(p-1)b^*$.  

Suppose $\op{ch}(\iota )+j<p$, then $\iota '=\iota -jb^*\in\mathfrak{A}^0$. 

Indeed,  $\op{ch}(\iota ')=\op{ch}(\iota )+j<p$ 
and $\kappa (\iota ')=\kappa (\iota )<u^*$.  
Therefore,  $w=t^{-\iota '}l\in\wt{\c N}^{sp}$, because 
$\iota '+\op{ch}(\iota ')b^*=
\iota -jb^*+(\op{ch}(\iota )+j)b^*<(s+m)b^*$, 
$\op{ch}(\iota ')\geqslant\op{ch}(\iota )\geqslant m$  
and $\kappa (\iota ')=\kappa (\iota )\leqslant (p-2)m+s$.  
\medskip 

If $\op{ch}(\iota )+j\geqslant p$ then 
(as earlier) $w=t^{jb^*}t^{-\iota }l\in t^{(\op{ch}(\iota )+j)
b^*}\wt{\c N}^{\dag }\subset\c I^{sp}$ 
and $\tilde w=0$.
The lemma is proved.  
\end{proof}

Clearly, we have the induced grading  
$\wt{\c N}^{ sp}=\oplus _{m\geqslant 0}\wt{\c N}^{ sp}[m]$, 
where $\wt{\c N}^{sp}[p-1]=0$. 
Any element from $\wt{\c N}^{sp}[m]$ appears as a sum of   
elements of the form $t^{-\iota }l$, where for some $s\geqslant 1$,  
$l\in\bar{\c L}^{\dag }(s)[m]_k$, 
$\iota +\op{ch}(\iota )b^*<(s+m)b^*$, $\op{ch}(\iota )\geqslant m$ 
and $\kappa (\iota )\leqslant (p-2)m+s$.

\begin{definition}  For $j\geqslant 0$ and $s\geqslant 1$, let:

a) $\wt{\c N}^{ sp}\langle j\rangle $ be the $k$-submodule  
in $\wt{\c N}^{ sp}$ generated by all $t^{-\iota }
l\in \wt{\c N}^{ sp}$ such that for some $m\geqslant 0$, 
$t^{-\iota }
l\in \wt{\c N}^{ sp}[m]$ and 
 $\op{ch}(\iota )\geqslant m+j$;
 
 b) $\wt{\c N}^{sp}(s, j\rangle $ be the submodule in 
$\wt{\c N}^{sp}\langle j\rangle $ 
 generated by $t^{-\iota }l$ (in the above notation) such that 
 $l\in C_s(\bar{\c L}^{\dag}_k)$. 
 \end{definition} 
 
 Note that:
 \medskip 

--- $\wt{\c N}^{sp}\langle 0\rangle =\wt{\c N}^{sp}(1, 0\rangle =\wt{\c N}^{sp}$;
\medskip 

--- all $\wt{\c N}^{ sp}\langle j\rangle $ and $\wt{\c N}^{sp}(s, j\rangle $ are 
ideals in $\wt{\c N}^{ sp}$;
\medskip 

--- for all $j_1,j_2$ and $s_1,s_2$, 
$[\wt{\c N}^{ sp}\langle j_1\rangle ,\wt{\c N}^{sp}\langle j_2\rangle ]
\subset\wt{\c N}^{ sp}\langle j_1+j_2\rangle $ and 
$[\wt{\c N}^{ sp}(s_1, j_1\rangle ,\wt{\c N}^{sp}(s_2, j_2\rangle ]
\subset\wt{\c N}^{ sp}(s_1+s_2, j_1+j_2\rangle $; 
\medskip 
  
--- $\wt{\c N}^{sp}\langle p-1\rangle =0$.  
 \medskip 
 
 --- for any $\iota\in\mathfrak{A}^0(p)$, 
 $t^{-\iota}D^{\dag }_{\iota 0}\in \wt{\c N}^{sp}$.
 \medskip

\subsection{The action $\Omega _{\gamma}$} \label{S2.5}
 Suppose $\gamma\in\Z /p$. 

If  
$\iota =p^n(q\alpha -(q-1)\beta )\in \mathfrak{A}^0$ 
and $t^{-\iota }l\in\wt{\c N}^{sp}\,$, where 
$l\in\bar{\c L}^{\dag }_k$, then by Lemma  \ref{L2.9}
$$\Omega _{\gamma }(t^{-\iota }l):=t^{-\iota }
\wt{\exp}(\gamma (p^n\alpha ) t^{b^*})l\,\in\wt{\c N}^{ sp}\, .$$

If $w\in\wt {\c N}^{ sp}$ then there is a unique presentation 
$w=\sum \limits _{\iota\in\mathfrak{A}^0}t^{-\iota }l_{\iota }$, 
where all $t^{-\iota }l_{\iota }\in\wt{\c N}^{ sp}$, and we 
set 
$$\Omega _{\gamma }(w)=\sum _{\iota\in\mathfrak{A}^0}
\Omega _{\gamma }(t^{-\iota }l_{\iota })\, .$$ 

The correspondence  
$w\mapsto \Omega _{\gamma }(w)$ is a 
well-defined action of the elements 
$\gamma $ of the (additive) group $\Z /p$ on 
the Lie algebra $\wt{\c N}^{ sp}$. This action is 
unipotent because for any $n\in \wt{\c N}^{sp}\langle j\rangle $, 
$\Omega _{\gamma }(n)\equiv n\,\op{mod}\,\wt{\c N}^{sp}\langle j+1\rangle $. 
\medskip   

Choose $\bar{\,e}^{sp}\in\bar{\c N}^{\dag }$ 
satisfying the following two conditions: 
\begin{equation} \label{E2.1} \ \ \bar{\,e}^{sp}\equiv 
\sum\limits _{\iota\in\mathfrak{A}^0(p)}t^{-\iota }D^{\dag}_{\iota 0}
\,\op{mod}\,C_2(\bar{\c L}_{\c K}) \,;
\end{equation} 
\begin{equation} \label{E2.2}
\wt{e}^{sp}:=\bar e^{sp}\,\op{mod}\,t^{(p-1)b^*}\bar{\c N}^{\dag }
 \in\wt{\c N}^{sp}\, .
\end{equation}
 
A choice of $\bar{\,e}^{\,sp}$ allows us to associate 
to the 
above defined action $\Omega _{\gamma }$ the  
\lq\lq conjugated\rq\rq\ action of 
$A^{\dag }_{\gamma }$ on $\bar{\c L}^{\dag }$ as follows.  

\begin{Prop} \label{P2.10}
 For any $\gamma \in\Z /p$, there are unique 
$\wt{c}_{\gamma}\in\wt{\c N}^{sp}\langle 1\rangle $ and 
 $A^{\dag } _{\gamma }\in \Aut _{\op{Lie}}\bar{\c L}^{\dag }$ such that 
 \medskip 
 
{\rm a)}\ $\sigma \tilde c_{\gamma }\in\wt{\c N}^{sp}\langle 1\rangle $ 
and $ \Omega _{\gamma } (\wt{\,e}^{\,sp})=
(\sigma \wt{c}_{\gamma })\circ (A^{\dag } _{\gamma }
\otimes\id _{\c K}) \wt{e}^{\,sp}\circ (- \wt{c}_{\gamma })$;
\medskip 

{\rm b)}\ for any $\iota\in\mathfrak{A}^{0}(p)$, 
$A^{\dag }_{\gamma }(D^{\dag }_{\iota 0})-D^{\dag }_{\iota 0}
\in \oplus _{m<\op{ch}(\iota )}\bar{\c L}^{\,\dag }[m]_k$.
\end{Prop}

\begin{proof} We need the following lemma. 

\begin{Lem} \label{L2.11}
Suppose $j,s\geqslant 1$ and $n\in\wt{\c N}^{ sp}(s,j\rangle $. Then there are 
unique $\c S(n),\c R(n)
\in\wt{\c N}^{sp}(s,j\rangle $ such that  
\medskip 

{\rm a)} $\c R(n)=\sum \limits_{\iota\in\mathfrak{A}^+(p)}
t^{-\iota }l_{\iota }$ with 
all\, $l_{\iota}\in C_s(\bar{\c L}^{\dag })_k$ 
{\rm (}if $\op{ch}(\iota )<j$ then $l_{\iota }=0${\rm )};
\medskip 

{\rm b)} $n=\c R(n)+(\sigma -\id )\c S(n)$. 
\end{Lem}

\begin{proof} [Proof of lemma] 
 Note that any $n\in\wt{\c N}^{ sp}(s, j\rangle $ appears as a 
sum of elements of the form 
$t^{-\iota }l$, where for some $m^0$ and $s^0$, it holds 
\linebreak 
$l\in\bar{\c L}^{\dag }(s^0)[m^0]_k\cap C_s(\bar{\c L}^{\dag })_k$, 
$\iota +\op{ch}(\iota )b^*<(s^0+m^0)b^*$,  
$\op{ch}(\iota )\geqslant m^0+j$ and $\kappa (\iota )\leqslant (p-2)m^0+s^0$. 
When proving the existence of $\c S(n)$ and $\c R(n)$ we can assume that 
$n=t^{-\iota }l$. 
\medskip

--- \ Let $\iota <0$\,. 
\medskip 

Set $\c R(n)=0$ 
and $\c S(n)=-\sum _{m\geqslant 0}
t^{-\iota p^m}\sigma ^ml$.

If $-\iota p^m\geqslant b^*(p-1)$ then 
$t^{-\iota p^m}\sigma ^ml\in t^{b^*(p-1)}\wt{\c N}^{\dag }=0$.

 If 
$-\iota p^m<b^*(p-1)$ then: 
\medskip  

--- \ \ $\iota p^m+\op{ch}(\iota p^m)b^*\leqslant 
\iota +\op{ch}(\iota )b^*<(s^0+m^0)b^*$;
\medskip  

--- \ \ $m^0=\op{ch}(l)=\op{ch}(\sigma ^ml)$ and 
$\op{ch}(\iota p^m)=\op{ch}(\iota )\geqslant m^0+j$;
\medskip 

--- \ \ $\kappa (\iota p^m)=\kappa (\iota )$.
\medskip 

Therefore, if $\iota <0$ then both 
$\c R(n),\c S(n)\in\wt{\c N}^{sp}(s,j\rangle $.  
\medskip

--- \ Let $\iota >0$\,.  
\medskip 

Suppose $p^{m(\iota )}$ is the maximal power of 
$p$ such that $\iota =p^{m(\iota )}\iota _1$ and 
$\iota _1\in\mathfrak{A}^0$. Then $\iota _1\in\mathfrak{A}^+(p)$: 
it will be sufficient to verify just 
the last inequality for $\kappa (\iota _1)=
\kappa (\iota )$ from the definition of 
$\mathfrak{A}^+(p)$ in Sect.\ref{S2.3}. Using  that 
$t^{-\iota }l\in \wt{\c N}^{sp}$, $w^0(\iota )\geqslant 1$  
and $\op{ch}(\iota )\geqslant m^0+1$ we obtain that 
$$(p-2)\op{ch}(\iota )+w^0(\iota )\geqslant (p-2)m^0+p-1
\geqslant (p-2)m^0+s^0\geqslant\kappa (\iota _1)\,.$$

Then we set 
$$\c R(n)=t^{-\iota _1}\sigma ^{-m(\iota )}l, \ \ \c S(n)=
\sum _{0\leqslant m<m(\iota )}\sigma ^m(\c R(n))\, .$$ 

Finally,  
if $0\leqslant m\leqslant m(\iota )$ then 
$\sigma ^m\c R(n)\in\wt{\c N}^{ sp}\langle j\rangle $. Indeed, 
\medskip

--- \ \ $\iota _1p^m+\op{ch}(\iota _1p^m)b^*\leqslant 
\iota +\op{ch}(\iota )b^*<
(s^0+m^0)b^*$; 
\medskip 

--- \  \ $\op{ch}(\iota _1p^m)=\op{ch}(\iota )\geqslant m^0+j$,  
$\sigma ^{-m(\iota )+m}l\in\bar{\c L}^{\dag }(s^0)[m^0]_k
\cap C_s(\bar{\c L}^{\dag})_k$\,;
\medskip 

--- \ \ $\kappa (\iota _1p^m)=\kappa (\iota )$.
\medskip 

So, we proved the existence of $\c R(n)$ and $\c S(n)$. 

The uniqueness follows from the fact 
that for $j\geqslant 1$, $\wt{\c N}^{sp}\langle j\rangle |_{\sigma =\id }=0$ 
and the appropriate $t^{-\iota }$ are independent modulo the subgroup 
$(\sigma -\id )\c K$, cf. Prop.\ref{P2.8}. The lemma is proved. 
\end{proof} 

Continue the proof of Prop.\ref{P2.10}. 
\medskip 

Use induction on $i\geqslant 1$ to prove the  
proposition modulo $\wt{\c N}^{sp}(i, i\rangle $. 
\medskip 

--- If $i=1$ take  $ \wt{c}_{\gamma }=0$, 
$A^{\dag }_{\gamma }=\id $ and use  
$\Omega _{\gamma }(\wt{e}^{\,sp })- \wt{e}^{\,sp}
\in \wt{\c N}^{ sp}(1, 1\rangle $. 
\medskip 

--- Assume $1\leqslant i<p$ and for  
$ \wt{c}_{\gamma }\in\wt{\c N}^{ sp}(1, 1\rangle $ 
and $A_{\gamma }^{\dag }\in\op{Aut}_{\op{Lie}}(\bar{\c L}^{\dag })$,  
$$H=\Omega _{\gamma } \wt{e}^{\,sp}- 
(\sigma  \wt{c}_{\gamma })\circ (A^{\dag } _{\gamma }
\otimes\id _{\c K})\wt{\,e}^{sp}\circ 
(- \wt{c}_{\gamma })\in \wt{\c N}^{ sp}(i, i\rangle \, .$$

Then $\c R(H),\c S(H)\in \wt{\c N}^{ sp}(i, i\rangle $. Set   
$\c R(H)=\sum\limits  _{\op{ch}(\iota )\geqslant i+m}t^{-\iota }H_{\iota m}$, where 
all $H_{\iota m}\in \wt{\c L}^{\dag }[m]_k\cap C_i(\bar{\c L}^{\dag })_k$.  
Introduce $A^{\dag \prime}_{\gamma }\in\Aut _{\op{Lie}}(\bar{\c L}^{\dag })$ 
by setting for all involved $\iota $ and $m$,  
$A_{\gamma }^{\dag \prime }(D^{\dag }_{\iota 0})=
A^{\dag }_{\gamma }(D^{\dag }_{\iota 0})-\sum _{m}H_{\iota m}$. Set also 
$\wt{c}^{\,\prime }_{\gamma }= \wt{c}_{\,\gamma }-{\c S}(H)$. Then 
$$\Omega _{\gamma } \wt{e}^{\,sp }\equiv  
(\sigma  \wt{c}^{\,\prime }_{\gamma 1})\circ (A^{\dag \prime } _{\gamma }
\otimes\id _{\c K})\wt{e}^{\,sp}\circ 
(- \wt{c}^{\,\prime }_{\gamma })\,\op{mod}\, \wt{\c N}^{ sp}(i+1, i+1\rangle \, .$$

The uniqueness follows similarly by induction on $i$ and the uniqueness part of 
Lemma \ref{L2.11}. 

The proposition is proved. 
\end{proof} 

We have obviously the following properties. 

\begin{Cor} \label{C2.12} 
 For any $\gamma ,\gamma _1\in\Z /p$, 
 
 {\rm a)}\ $A^{\dag } _{\gamma +\gamma _1}=
A^{\dag } _{\gamma }A^{\dag } _{\gamma _1}$;
 
 {\rm b)}\ $\Omega _{\gamma }(\wt{c}_{\gamma _1})
\circ (A^{\dag } _{\gamma _1}\otimes\id _{\c K}) {\wt c}_{\gamma }=
  \wt{c}_{\gamma +\gamma _1}$;
 
 {\rm c)}\ if $l\in\bar{\c L}^{\dag }[m]$ then 
$A^{\dag }_{\gamma }(l)-l\in \oplus _{m'<m}
\bar{\c L}^{\dag }[m']$, e.g.  
$A^{\dag } _{\gamma }|_{\bar{\c L}^{\dag}[0]}=\id \, .$
\end{Cor}

\subsection{The action $\Omega _U$} \label{S2.6}

Let $A^{\dag }:=A^{\dag }_{\gamma }|_{\gamma =1}$. 
Then for any $\gamma =n\,\op{mod}\,p$,   
$A^{\dag }_{\gamma }=A^{\dag n}$, in particular,  
$A^{\dag p}=\id _{\bar{\c L}^{\dag }}$. 
By part c) of the above corollary, 
for all $m\geqslant 0$, 
$$(A^{\dag }_{\gamma }-
\id _{\bar{\c L}^{\dag }})\left (
\oplus _{m'\leqslant m}\bar{\c L}^{\dag }[m']\right )
\subset\oplus_{m'<m}\bar{\c L}^{\dag }[m']\, .$$ 
Therefore, 
there is a differentiation $B^{\dag }\in
\op{End} _{\op{Lie}}\bar{\c L}^{\dag }$ such that 
for all $m\geqslant 0$, 
$B^{\dag }(\bar{\c L}^{\dag }[m])
\subset\oplus _{m'<m}\bar{\c L}^{\dag }[m']$ and for all $\gamma\in\Z /p$, 
$A^{\dag } _{\gamma }=\wt{\exp }(\gamma B^{\dag })$.

Recover this derivation by applying the methods from  \cite{Ab11}, Sect.3. 

Namely, define a coaction of the formal finite group scheme 
$\alpha _p=\op{Spec} \F _p[U]$ on $\wt{\c N}^{sp}$ as follows. (Here $U^p=0$ and 
the coaddition is such that $\Delta U=U\otimes 1+1\otimes U$.)  

If  
$\iota =p^n(q\alpha -(q-1)\beta )\in \mathfrak{A}^0$ 
and $t^{-\iota }l\in\wt{\c N}^{sp}\,$, where 
$l\in\bar{\c L}^{\dag }_k$, set 
$$\Omega _{U}(t^{-\iota }l):=t^{-\iota }
\wt{\exp}( U\otimes (p^n\alpha ) t^{b^*})l\,\in \F _p[U]
\otimes \wt{\c N}^{ sp}\, .$$ 

As a result, 
\begin{equation} \label{E2.3}
\Omega _U(\wt{e}^{\,sp})=\sigma (\wt c_U)
\circ (A^{\,\dag }_U\otimes\id )\wt{e}^{\,sp}\circ (-\wt c_U)\, ,
\end{equation} 
where for all $\gamma\in\Z /p $, 
$A^{\dag }_U=\wt{\exp}(UB^{\dag })$ and 
$\wt{c}_U|_{U=\gamma }= \wt{c}_{\gamma }$. 
In \cite{Ab12} we also established that 
\medskip 

--- $\wt{c}_U= \wt{c}^{\,(1)}U+\dots +
 \wt{c}^{\,(p-1)}U^{p-1}$, where all $c^{(j)},\sigma c^{(j)}
 \in\wt{\c N}^{sp}\langle j\rangle $;   
\medskip

--- the cocycle $ \wt{c}_U$ is determined uniquely by its 
linear part $\wt{c}^{\,(1)}$; 
\medskip 

--- the action $\Omega _U=\sum _{0\leqslant i<p}\Omega ^{\,i}U^i$ 
(here $\Omega ^0=\id $) is recovered uniquely from its 
differential $d\,\Omega _U:=\Omega ^1U$.

\subsection{Ideals $\bar{\c L}^{\,\dag }[v_0]$ and $\bar{\c L}[v_0]$} 
\label{S2.7} \ \ 

Recall that  $\bar{\c L}^{\,\dag }[0]$ is the minimal Lie 
subalgebra of $\bar{\c L}^{\,\dag }$ such that  
$\bar{\c L}^{\,\dag }[0]_k$ contains all $D^{\dag }_{\iota n}$ 
with $\iota\in\mathfrak{A}^0_0(p)=\{\iota\in\mathfrak{A}^0(p)\ |\ 
\op{ch}(\iota )=0\}$. 
Then $\bar{\c L}^{\,\dag }[0]$ has the induced filtration 
$\{\bar{\c L}^{\,\dag }(s)[0]\}_{s\geqslant 1}$ and 
there is epimorphism of filtered Lie algebras 
$\c V^{\,0}:\bar{\c L}^{\dag }\To\bar{\c L}^{\dag }[0]$ such that 
$D^{\dag }_{\iota n}\mapsto D^{\dag }_{\iota n}$ if 
$\iota\in\mathfrak{A}^0_0(p)$ and $D^{\dag }_{\iota n}\mapsto 0$, otherwise. 

By Lemma \ref{L2.6}, 
$\mathfrak{A}^0_0(p)=\{qa\ |\ a\in [0,(p-1)v_0)\cap \Z ^+(p)\}$. By Remark from 
the proof of Proposition \ref{P2.8}a), 
 the correspondences 
$D^{\dag }_{qa,n}\mapsto D_{an}$ establish 
isomorphism of filtered Lie algebras $j^0:\bar{\c L}^{\dag }[0]\To \bar{\c L}$. 

Let $\c V:=j^0\c V^0:\bar{\c L}^{\dag }\To\bar{\c L}$.

Define the ideal $\bar{\c L}^{\,\dag }[v_0]$ as the minimal 
ideal in $\bar{\c L}^{\,\dag }$ containing all 
$A^{\dag }_{\gamma }(\Ker\,\c V)$, $\gamma\in\Z /p$.  
Set $\bar{\c L}[v_0]=\c V(\bar{\c L}^{\,\dag }[v_0])$.  
Then $\bar{\c L}[v_0]$ is the minimal ideal in $\bar{\c L}$ such that 
$\c V^{-1}(\bar{\c L}[v_0])$ is invariant with respect to all $A^{\,\dag }_{\gamma }$. 

\begin{remark} $\Ker\c V_k$ is the ideal in $\bar{\c L}_k^{\dag }$ generated by the elements 
$D^{\dag }_{\iota n}$ such that $\op{ch}\iota \geqslant 1$. By Prop.\,\ref{P2.8}a)  
for all such $\iota $, 
$\iota p^{-v_p(\iota )}\geqslant (p-1)v_0$. In particular, all $D_{\iota p^{-v_p(\iota )},n}
\in{\c L}(p)_k$. Therefore, the projection $\bar{\op{pr}}:\c L\To\bar{\c L}$ 
factors through an epimorphic map of 
Lie algebras $\bar{\eta }^{\dag }:\c L\To \bar{\c L}^{\dag }$. In particular, 
$\bar{\op{pr}}^{-1}\bar{\c L}[v_0]=(\bar{\eta }^{\dag })^{-1}\bar{\c L}^{\dag }[v_0]$.
\end{remark}

\begin{Prop} \label{P2.13} 
If $l\in\bar{\c L}^{\,\dag }$ and 
$\gamma\in\Z /p$ then 
$$\c V(A^{\dag }_{\gamma }l)
 \equiv \c V(l)\,\op{mod}\,\bar{\c L}[v_0]\,.$$
\end{Prop}

\begin{proof} a) Let $l'=\c V^0(l)$. Then $l\in l'+\op{Ker}\c V$ 
and, therefore, 
 $$A^{\dag }_{\gamma }(l)\in A^{\dag }_{\gamma }(l')+
 A^{\dag }_{\gamma }(\Ker\c V)\subset l'+\c V^{-1}\bar{\c L}[v_0]\, .$$
 It remains to apply $\c V$ to this embedding. 
(Use that $A^{\dag }_{\gamma }|_{\op{Im}\,\c V^0}=\id $.)
\end{proof}

The ideal $\bar{\c L}[v_0]$ can be also defined in terms related to 
the  action  $\Omega _U$. 
If $B^{\dag }$ is the differentiation 
from Sect.\ref{S2.6} then 
$\bar{\c L}[v_0]$ appears as the minimal ideal in $\bar{\c L}$ such that  
$\bar{\c L}[v_0]_k$ contains all the elements 
$\c VB_k^{\dag }(D^{\dag }_{\iota 0})$, 
where $\iota \in \mathfrak{A}^0(p)$ and $\op{ch}(\iota )\geqslant 1$ 
(if 
$\iota\in\mathfrak{A}^0_0(p)$ then 
$B^{\dag }(D^{\dag }_{\iota 0})=0$). This is implied by the following proposition. 

\begin{Prop} \label{P2.14} 
Suppose ${\c I}$ is an ideal in $\bar{\c L}$. Then 
 the following conditions are equivalent: 
 \medskip 
 
 {\rm a)}\ for any $\gamma\in\Z /p$, $A^{\dag }_{\gamma }(\Ker \c V)\subset 
\c V^{-1}({\c I})$;
\medskip 

{\rm b)}\ $B^{\dag }(\Ker\c V)\subset \c V^{-1}({\c I})$. 
\end{Prop}
 
\begin{proof} 

Part a) implies b) because for any $l\in \bar{\c L}^{\dag }$ 
we have a non-degenerate system of linear relations 
\begin{equation} \label{E2.4} 
(A^{\dag }_{\gamma }-\id _{\bar{\c L}^{\dag }})l\equiv 
\sum _{1\leqslant s<p}\gamma ^sB^{\dag s}(l)/s!
\,\op{mod}\,{\c I}
\end{equation} 
with $\gamma =1,\dots,p-1$. 

Vice versa, b) 
implies that for all $s\geqslant 1$, 
$B^{\dag s}(\Ker {\c V})\subset B^{\dag }(\Ker\,\c V)$. 
Indeed, 
$\bar{\c L}^{\dag }=\Ker\c V\oplus\bar{\c L}^{\dag }[0]$ implies that 
$\c V^{-1}(\c I)=\Ker\c V\oplus (j^0)^{-1}(\c I)$. 
Therefore, $B^{\dag\,2}(\Ker\,\c V)\subset B^{\dag }(\c V^{-1}(\c I))
= B^{\dag }(\Ker\c V)$ (use 
$B^{\dag }|_{\bar{\c L}^{\dag }[0])}=0$). It remains to use 
relations \eqref{E2.4}. 
Proposition is proved. 
\end{proof}

\subsection {Lie algebras $\c N^{(q)}$, 
$\bar{\c N}^{(q)}$ and $\wt{\c N}^{(q)}$}   \label{S2.8} 
\ \

Introduce an analogue $\c N^{(q)}\subset\c L_{\c K}$ of $\c N$ 
as the $k$-module generated by all $t^{-a}l$, where for some 
$s\geqslant 1$, 
$l\in\c L(s)_k$ and  $a<sb^*$. It is a Lie $k$-algebra and 
$e^{(q)}$ together with all  
$\gamma *e^{(q)}$, $\gamma\in\Z /p$, cf. Sect.\ref{S1.5},  belong to $ \c N^{(q)}$. 
\medskip 

Similarly, introduce the Lie algebras $\bar{\c N}^{(q)}$ 
(use the algebra $\bar{\c L}$ 
instead of $\c L$) 
and $\wt{\c N}^{(q)}=\bar{\c N}^{(q)}/t^{(p-1)b^*}\bar{\c N}^{(q)}$.  
These algebras are related to $\c N^{(q)}$ via the natural projection 
$\bar{\op{pr}}_{\c K}:\c L_{\c K}\To\bar{\c L}_{\c K}$. 
The appropriate images of $e^{(q)}$ in $\bar{\c N}^{(q)}$ and $\wt{\c N}^{(q)}$ 
will be denoted, resp.,  by 
$\bar e^{(q)}$ and $\wt{e}^{(q)}$. 
Note that 
there are natural identifications 
$\bar{\c N}^{(q)}=\c V_{\c K}(\bar{\c N}^{\dag })$ and 
$\wt{\c N}^{(q)}=\c V_{\c K}(\wt{\c N}^{\dag })$, where 
$\bar{\c N}^{\dag }$, $\wt{\c N}^{\dag }$ and  
$\c V_{\c K}$ were defined in Sect.\ref{S2.4}. 
\medskip 

\subsection{Generators of $\bar{\c L}[v_0]$} \label{S2.9}

Introduce the following condition of compatibility
\begin{equation} \label{E2.5} 
 \c V_{\c K}(\bar e^{sp})=\bar e^{(q)}\, .
\end{equation}

By Prop.\ref{P2.14}, $\bar{\c L}[v_0]$ is the minimal ideal in 
$\bar{\c L}$ such that  
for all $\iota\in\mathfrak{A}^0(p)$ with $\op{ch}(\iota )\geqslant 1$, 
$\c V_kB^{\dag }_k(D^{\dag }_{\iota 0})      
\in \bar{\c L}[v_0]_k$. (Note that this implies 
$\c VB^{\dag }(C_2(\bar{\c L}^{\dag }))\subset [\bar{\c L}[v_0],\bar{\c L}]$.)

\begin{Prop} \label{P2.15}
 If $\op{ch}(\iota )\geqslant 2$ then $\c V_kB^{\dag }_k(D^{\dag }_{\iota 0})
\in [\bar{\c L}[v_0],\bar{\c L}]_k$. 
\end{Prop}

\begin{proof} 
 Suppose  $\tilde{e}^{\,(q)}=\sum \limits _{a}t^{-qa}l_a^{(q)}$ 
and  $\tilde{e}^{\,sp}=
\sum \limits _{\iota}t^{-\iota }l_{\iota }^{sp}$, where  all  
$l_a^{(q)}\in\bar{\c L}_k$ and  $l_{\iota }^{sp}\in\bar{\c L}^{\dag }_k$. 
Note that:
\medskip 

--- if $a\in\Z ^0(p)$ then  $l_a^{(q)}\equiv D_{a0}\,\op{mod}\,C_2(\bar{\c L})_k$;
\medskip 

--- if $a\notin \Z ^0(p)$ then  $l_a^{(q)}\in C_2(\bar{\c L})_k$;
\medskip 
 
--- if $\iota =qa$ with $a\in\Z $, then $\c V_k(l^{sp}_{\iota })=l_a^{(q)}$, 
otherwise $\c V_k(l_{\iota }^{sp})=0$;
\medskip 

--- if $\iota\in\mathfrak{A}^0(p)$ 
then $l^{sp}_{\iota }\equiv D^{\dag }_{\iota 0}\,\op{mod}\,
C_2(\bar{\c L}^{\dag })_k$, otherwise, $l_{\iota }^{sp}\in C_2(\bar{\c L}^{\dag })_k$.
\medskip 

Applying formalism from Sect.\ref{S2.6} we obtain 
\begin{equation} \label{E2.6} 
\Omega _U(\wt{e}^{\,sp})\equiv 
(U\sigma \tilde{c}^{\,1})\circ (\wt{e}^{\,sp}+U(B^{\dag }
\otimes\id _{\c K})\wt{e}^{\,sp})
\circ (-U\tilde{c}^{\,1})\,\op{mod}\,U^2\wt{\c N}^{\dag }\,,
\end{equation} 
where $\tilde{c}^{\,1},\sigma\tilde c^1\in\wt{\c N}^{sp}\langle 1\rangle $. 
Note that 
$$\Omega _U(\wt{e}^{\,sp})\equiv \wt{e}^{\,sp}+
U\sum _{\iota }(p^m\alpha )_{\iota }t^{-\iota +b^*}l_{\iota }^{sp}
\,\op{mod}\,U^2\, ,$$ 
where 
$\iota =q(p^m\alpha )_{\iota }-(q-1)(p^m\beta )_{\iota }$. 
\medskip  

Applying $\c V_{\c K}$ to relation \eqref{E2.6} and setting 
$\tilde x:=\c V_{\c K}\tilde c^1$ we obtain 
\begin{equation} \label{E2.7}  
\tilde{e}^{(q)}+U\sum _{a} at^{-qa +b^*}l_{a}^{(q)}
\equiv 
\end{equation}
$$(U\sigma\tilde{x})\circ \left (\tilde{e}^{(q)}+U\sum\limits _{\iota }t^{-\iota }
\c V_kB^{\dag }_k(l_{\iota }^{\,sp})\right )\circ 
(-U\tilde x)\,\op{mod}\,U^2\tilde{\c N}^{(q)}\, .$$

Let $\tilde f_1, \tilde f_2\in\wt{\c N}^{\,(q)}$ be such that 
$$ (U\sigma \tilde{x})\circ \tilde{e}^{\,(q)}\equiv 
\tilde{e}^{\,(q)}+U(\sigma \tilde{x}+\tilde{f}_1)
\,\op{mod}\,U^2$$
$$ \tilde{e}^{\,(q)}\circ (-U\tilde{x})\equiv 
\tilde{e}^{\,(q)}+U(- \tilde{x}+\tilde{f}_2)
\,\op{mod}\,U^2\, .$$
There are explicit formulas for $\tilde f_1$ and $\tilde f_2$, cf. e.g. Sect.3.2 of 
\cite{Ab11}, but 
we need only that they are just $\F _p$-linear combinations 
of the commutators $[\dots [\sigma \tilde{x}, 
\tilde{e}^{\,(q)}],\dots ,\tilde{e}^{\,(q)}]$ and, resp.,  
$[\dots [\tilde{x}, 
\tilde{e}^{\,(q)}],\dots ,\tilde{e}^{\, (q)}]$. 

Comparing the coefficients for $U$ in \eqref{E2.7} we obtain 
\begin{equation} \label{E2.8} \sum _{a} at^{-qa +b^*}l_{a}^{(q)}
=\sigma\tilde{x}-\tilde{x}+\sum\limits _{\iota }t^{-\iota }
\c V_kB^{\dag }_k(l_{\iota }^{\,sp})+\tilde f_1+\tilde f_2\,.
\end{equation}

Note that:
\medskip 

a) $\tilde{c}^{\,1},\sigma\tilde c^1\in\wt{\c N}^{\,sp}\langle 1\rangle $ implies that 
$\tilde x,\sigma \tilde x\in \sum\limits _{\op{ch}(\iota )\geqslant 1}
t^{-\iota }\bar{\c L}_k$;
\medskip 

b) $\{\ y\in\sum\limits _{\op{ch}(\iota )\geqslant 1}
t^{-\iota }\bar{\c L}_k\ |\ \sigma y=y\}=0$;

c) if $\iota\not\in\mathfrak{A}^0(p)$ then $\c V_kB_k^{\dag }(l_{\iota }^{sp})\in 
[\bar{\c L}[v_0],\bar{\c L}]_k$;
\medskip 

d) if $\iota\in\mathfrak{A}^0(p)$ then $\c V_kB^{\dag }_k(l_{\iota }^{sp})
\equiv \c V_kB^{\dag }_k(D^{\dag }_{\iota 0})
\,\op{mod}\,[\bar{\c L}[v_0],\bar{\c L}]_k$. 
\medskip

Let $\tilde x= \sum\limits _{\iota } 
t^{-\iota }x_{\iota }$ 
and $\tilde f_1+\tilde f_2= \sum\limits  _{\iota }t^{-\iota }f_{\iota }$,  
where $x_{\iota },f_{\iota }\in\bar{\c L}_k$ and 
the both sums are taken for  
$\iota \in\mathfrak{A}^0$ such that $\op{ch}(\iota )\geqslant 1$.
Let $\tilde x[m]$ be a part of the first sum 
containing all the summands $t^{-\iota }x_{\iota }$ with $\op{ch}(\iota )=m$. 
Similarly, define a part $\tilde f[m]$ of the second sum. Note that 
$\tilde f[m]$ is a linear combination of the commutators 
$[\dots [\tilde x[m],\tilde e^{(q)}],\dots ,\tilde e^{(q)}]$ 
and $[\dots [\sigma (\tilde x[m]),\tilde e^{(q)}],\dots ,\tilde e^{(q)}]$.
\medskip 

Then \eqref{E2.8} and above congruence d) imply that for any $m\geqslant 2$, 
$$\c H[m]\equiv -\tilde f[m]\in \,\op{mod}\,[\bar{\c L}[v_0],\bar{\c L}]_{\c K}\, ,$$ 
where 
$\c H[m]:=\sigma (\tilde x[m])-\tilde x[m]+\sum\limits _{\op{ch}(\iota )=m}t^{-\iota }
\c V_kB^{\dag }_k(D^{\dag }_{\iota 0})$.

Let $\bar{\c D}(s):=[\bar{\c L}[v_0],\bar{\c L}]+\bar{\c L}(s)$.

Prove by induction on $s\geqslant 1$ that  $\tilde x[m]\in\bar{\c D}(s)_{\c K}$ and  
$\c V_kB_k^{\dag }(D^{\dag}_{\iota 0})\in\bar{\c D}(s)_k$ (here $\op{ch}(\iota )=m$).
\medskip 

If $s=1$ there is nothing to prove. 

Suppose it is proved for $s<p$. 

Then  
$\tilde f[m]\in\bar{\c D}(s+1)_{\c K}$ and, therefore,  
$\c H[m]\in \bar{\c D}(s+1)_{\c K}$. Then analog of Lemma \ref{L2.11} implies that 
$\tilde x[m]$ and  $\sum\limits _{\op{ch}(\iota )=m}t^{-\iota }
\c V_kB^{\dag }_k(D^{\dag }_{\iota 0})$ belong to $\bar{\c D}(s+1)_{\c K}$. In particular, 
all $\c V_kB_k^{\dag }(D_{\iota 0})\in\bar{\c D}(s+1)_k$. 

The proposition is proved because $\bar{\c D}(p)=[\bar{\c L}[v_0],\bar{\c L}]$.
\end{proof} 

\begin{Cor} \label{C2.16} 
$\bar{\c L}[v_0]$ is the minimal ideal in 
$\bar{\c L}$ such that  
for all $\iota\in\mathfrak{A}_1^0(p):=
\{\iota \in\mathfrak{A}^0(p)\ |\ \op{ch}(\iota )=1\}$, 
$\c V_kB^{\dag }_k(D^{\dag }_{\iota 0})
\in \bar{\c L}[v_0]_k$. 
\end{Cor}

\section{Application to the ramification filtration} \label{S3}

\subsection{Statement of the main result} \label{S3.1} 

Recall that in Sect.\ref{S1} we fixed an element $e\in\c L_{\c K}$ satisfying 
conditions \eqref{E1.1} and \eqref{E1.3}. We also fixed $f\in\c L_{sep}$  such that 
$\sigma f=e\circ f$
and introduced epimoirphism $\eta _e=\pi _f(e):\c G\To G(\c L)$ 
which induces identification 
$\c G_{<p}\simeq G(\c L)$. Conditions 
\eqref{E1.1} and \eqref{E1.3} mean that $\eta _e$ is 
a \lq\lq sufficiently good\rq\rq\ 
lift of the reciprocity map of class field theory. 
We are going to describe the ideal $\c L^{(v_0)}$ of $\c L$ such that 
$\eta _e(\c G^{(v_0)})=\c L^{(v_0)}$  
via  the ideal $\bar{\c L}[v_0]$ introduced in Sect.\,\ref{S2}. 
In the next section this result will be related  
to the explicit 
description of $\c L^{(v_0)}$ from \cite{Ab1}.

Consider the parameters $\delta _0, r^*, N^*, q$ from Sect.\ref{S2} 
(they depend just on the original $v_0>0$). 
Note that if $e^{(q)}=\sigma ^{N^*}(e)$ and $f^{(q)}=\sigma ^{N^*}(f)$ then 
the appropriate morphism $\pi _{f^{(q)}}(e^{(q)})$ coincides with $\eta _e$. 

The 
ideal $\bar{\c L}[v_0]\subset\bar{\c L}$ was defined in the terms 
of action of 
the formal group $\alpha _p$ on  
 $\wt{e}^{\,sp}=\bar{e}^{sp}\,\op{mod}\,t^{(p-1)b^*}\bar{\c N}^{\dag }
\in\wt{\c N}^{sp}\subset\wt{\c N}^{\dag }$ 
which satisfies assumption \eqref{E2.1}. In Sect.\,\ref{S2.8} we introduced 
compatibility condition \eqref{E2.5} relating  
the elements $e$ and $\bar{e}^{sp}$. This condition can be 
definitely satisfied if e.g.  
$e=\sum\limits  _{a\geqslant 0}t^{-a}l_a, 
\text{\  where\   all\  }l_a\in\c L_k.$

\begin{Thm} \label{T3.1} Under condition \eqref{E2.5},   
$\c L^{(v_0)}=\bar{\op{pr}}^{-1}\bar{\c L}[v_0]$. 
\end{Thm}

\begin{remark}
 According to Remark from Sect.\,\ref{S2.7} this theorem also states that 
 there is epimorphism of Lie algebras 
 $\bar{\eta }^{\dag }:\c L\To \bar{\c L}^{\dag }$ such that 
 $\c L^{(v_0)}=(\bar{\eta }^{\dag })^{-1}\bar{\c L}^{\dag }[v_0]$. 
\end{remark}

\subsection{Inductive assumption} \label{S3.2} 

Prove theorem by induction on $s_0\geqslant 1$ in the following form 
(the statement of theorem appears with $s_0=p$)
\begin{equation}  \label{E3.3} 
\c L^{(v_0)}+C_{s_0}(\c L)=\bar{\op{pr}}^{-1}(\bar{\c L}[v_0])+\c L(s_0)\, .
\end{equation}

It is obviously true for $s_0=1$ because $C_1(\c L)=\c L(1)=\c L$.

Suppose \eqref{E3.3} holds for some $1\leqslant s_0<p$. 

Note that  for any $\gamma\in\Z /p$, the image $\gamma *\wt{e}^{(q)}$ of $\gamma *\bar e^{(q)}$ 
in  $\wt{\c N}^{(q)}$ coincides with 
$\c V_{\c K}
\Omega _{\gamma }(\wt{e}^{\,sp })$. 

\begin{Prop} \label{P3.2} 
There is $x_{\gamma}\in t^{b^*}\hskip -5pt
\sum\limits _{1\leqslant s<s_0}t^{-sb^*}\c L(s)_{\m }\subset\c N^{(q)}$ 
such that  
 $$\gamma *e^{(q)}\equiv (\sigma x_{\gamma })\circ e^{(q)}
 \circ (-x_{\gamma })\,\op{mod}\,(\c L^{(v_0)}+C_{s_0}(\c L))_{\c K}\, .$$ 
\end{Prop}
 
\begin{proof} From Prop.\ref{P2.10}\,a) 
it follows that 
$$\gamma *\wt e^{\,(q)}=
(\sigma\wt x_{\gamma })\circ (\c VA^{\dag }_{\gamma }\otimes
\id _{\c K})\wt e^{\,sp}\circ (-\wt x_{\gamma })
\, ,$$
where $\wt x_{\gamma }=\c V_{\c K}(\wt c_{\gamma })\in 
\c V_{\c K}(\wt{\c N}^{\,sp}\langle 1\rangle )\subset 
\c V_{\c K}(t^{b^*}\wt{\c N}^{\dag })=
t^{b^*}\wt{\c N}^{(q)}$ (use Remark before Lemma \ref{L2.9}).
\medskip 

Recall that $\bar e^{\,sp }\in\bar{\c N}^{\,\dag }$ is a lift of $\wt{e}^{\,sp}\in\wt{\c N}^{\,sp}\subset 
\wt{\c N}^{\,\dag }$ such that $\c V_{\c K}(\bar e^{\,sp })=\bar e^{(q)}$ and  
$\sigma $ is nilpotent on the kernel of the projection 
$\bar{\c N}^{\,\dag }\To\wt{\c N}^{\,\dag }$. 
Therefore, 
proceeding similarly to the proof of Prop.\ref{P1.4} we can establish the 
existence of a unique lift $\bar x_{\gamma }\in t^{b^*}\bar{\c N}^{(q)}$ 
of $\tilde x_{\gamma }$ such that 
\begin{equation} \label{E3.4} \gamma *\bar e^{\,(q)}=
(\sigma \bar x_{\gamma })\circ (\c VA^{\dag }_{\gamma }\otimes\id _{\c K})\bar e^{\,\dag }
\circ (-\bar x_{\gamma })\, .
\end{equation}

Prop.\ref{P2.13} implies that 
$(\c VA^{\dag }_{\gamma }\otimes\id _{\c K})\bar e^{\dag }
\equiv \bar e^{(q)}\,\op{mod}\,\bar{\c L}[v_0]_{\c K}$ and we obtain 
the following congruence 
\begin{equation} \label{E3.5} \gamma *e^{(q)}\equiv (\sigma x_{\gamma })
\circ e^{(q)}\circ (-x_{\gamma})\,\op{mod}\,
\bar{\op{pr}}^{-1}\bar{\c L}[v_0]_{\c K}\, ,
\end{equation} 
where $x_{\gamma }\in\c L_{\c K}$ is
any lift of $\bar x_{\gamma }$. 

We can choose 
$x_{\gamma}\in t^{b^*}
\sum _{1\leqslant s<s_0}t^{-sb^*}\c L(s)_{\m }$ 
when taking this congruence modulo the ideal 
$(\bar{\op{pr}}^{-1}\bar{\c L}[v_0]+\c L(s_0))_{\c K}$. 
It remains to use the inductive assumption. 
The proposition is proved.
\end{proof}

\begin{remark} a)Due to the criterion from Sect.\ref{S1.5} 
congruence \eqref{E3.5} already implies that 
$\c L^{(v_0)}\subset\bar{\op{pr}}^{-1}\bar{\c L}[v_0]$ 
(use that all $x_{\gamma }$ are defined over $\c K$). 

b) In the above proof we have automatically that 
$\sigma\bar x_{\gamma }\in t^{b^*}\bar{\c N}^{(q)}$ and 
$\sigma x_{\gamma}\in t^{b^*}
\sum _{1\leqslant s<s_0}t^{-sb^*}\c L(s)_{\m }$.
\end{remark}

For (non-commuting) variables $U$ and $V$ from some Lie $\F _p$-algebra $L$ of 
nilpotent class $<p$, let $\delta ^0(U,V):=U\circ V-(U+V)$. 
Note, if $U$ and $V$ are well-defined modulo $C_{s_0}(L)$ then 
$\delta ^0(U,V)$ is well-defined modulo $C_{s_0+1}( L)$. 
\medskip 

Let $y_{\gamma }=\gamma *e^{(q)}-e^{(q)}+
\delta ^0(\gamma *e^{(q)},x_{\gamma })-
\delta ^0(\sigma x_{\gamma },e^{(q)})$.

\begin{Lem} \label{L3.3}  
 For any $\gamma\in \Z /p$,  there is 
 $X_{\gamma }\in\c L_{sep}$ such that 
 \medskip 
 
 {\rm a)}\ $\gamma *e^{(q)}\equiv  (\sigma X_{\gamma })\circ 
 e^{(q)}\circ (-X_{\gamma })\op{mod}
 \,([\c L^{(v_0)},\c L]+
C_{s_0+1}(\c L))_{sep}\,;$
\medskip 

{\rm b)}\ $X_{\gamma }\equiv x_{\gamma }\,
\op{mod}\,(\c L^{(v_0)}+C_{s_0}(\c L))_{sep}.$
\end{Lem}

\begin{proof}[Proof of lemma] Prop.\ref{P3.2} implies that  
$$y_{\gamma }\equiv \sigma x_{\gamma }-x_{\gamma }\,\op{mod}\,
(\c L^{(v_0)}+C_{s_0}(\c L))_{\c K}\, .$$
Therefore, there is $X_{\gamma }\in\c L_{sep}$ such that 
$\sigma X_{\gamma }-X_{\gamma }=y_{\gamma }$ and $X_{\gamma }$ 
satisfies the congruence b).  

It remains to note that 
a) is 
equivalent to the following congruence  
$$
\sigma X_{\gamma }-X_{\gamma }\equiv 
\gamma *e^{(q)}-e^{(q)}+
\delta ^0(\gamma *e^{(q)},X_{\gamma })-
\delta ^0(\sigma X_{\gamma },e^{(q)})\,$$ 
modulo $([\c L^{(v_0)},\c L]+
C_{s_0+1}(\c L))_{sep}$ and by the same modulo we have 
$\delta ^0(\gamma *e^{(q)},X_{\gamma })\equiv 
\delta ^0(\gamma *e^{(q)}, x_{\gamma })$ and  
$\delta ^0(\sigma X_{\gamma },e^{(q)})\equiv 
\delta ^0(\sigma x_{\gamma }, e^{(q)})$.  
\end{proof}
 
The element  $y_{\gamma }$ 
can be uniquely written as 

$$y_{\gamma }=
\sum _{\substack{m\geqslant 0\\ a\in\Z ^+(p)}}t^{-ap^m}l_{am}+l_O\,,$$
where all $l_{am}\in\c L_k$ and $l_O\in\c L_{O}$ (and $O=k[[t]]\subset\c K$). 
By Prop.\ref{P1.3} the ideal 
$\c L^{(v_0)}+C_{s_0+1}(\c L)$ appears as 
the minimal ideal in the set of all ideals $\c I$ such that:  
\medskip 

--- $\c I\supset [\c L^{(v_0)},\c L]+C_{s_0+1}(\c L)$;
\medskip 

---  if $a\in\Z ^+(p)$ and $a\geqslant qv_0-b^*$ then 
$l^{(a)}:=\sum _{m\geqslant 0}\sigma ^{-m}l_{am}\in\c I_k$. 
\medskip

\begin{Prop}\ \label{P3.4} 
$\c L(s_0+1)\subset \c L^{(v_0)}+C_{s_0+1}(\c L)$, or (equivalently) if 
$a\geqslant s_0v_0$ then all $D_{an}\in\c L^{(v_0)}_k+C_{s_0+1}(\c L)_k$.
\end{Prop} 

\begin{proof} We have 
 $e^{(q)}, \gamma *e^{(q)}\in\c N^{(q)}$ and   
$\gamma *e^{(q)}-e^{(q)}\in t^{b^*}\c N^{(q)}$. 
Then from Prop.\ref{P3.2} 
we obtain that   
$y_{\gamma }\equiv\gamma *e^{(q)}-e^{(q)}$ 
modulo 
$$
[t^{b^*}\c N, 
\c N]\subset (t^{b^*}\c N)\cap C_2(\c L)_{\c K}\subset $$
$$\subset \sum _{2\leqslant s_1+s_2\leqslant s_0}[\c L(s_1), 
\c L(s_2)]_{\m }t^{-(s_1+s_2-1)b^*}+
\c L(s_0+1)_{\c K}\cap C_2(\c L)_{\c K}$$
$$\subset t^{-(s_0-1)b^*}\c L_{\m }+\c L(s_0+1)_{\c K}\cap C_2(\c L)_{\c K}.$$

\begin{Lem} \label{L3.5}
 $\c L(s_0+1)\cap C_2(\c L)\subset \c L^{(v_0)}\cap C_2(\c L)+C_{s_0+1}(\c L)$. 
\end{Lem}

\begin{proof}[Proof of lemma] From the 
definition of $\c L(s_0+1)$  it follows that  
the $k$-module 
$\c L(s_0+1)_k\cap C_2(\c L)_k$ is generated 
by the commutators 
$$[\dots [D_{a_1n_1},D_{a_2n_2}],\dots ,D_{a_rn_r}]$$ 
such that $r\geqslant 2$ and 
$\op{wt} (D_{a_1n_1})+\dots +\op{wt}(D_{a_rn_r})\geqslant s_0+1$. 

Here for $1\leqslant i\leqslant r$, 
$\op{wt} (D_{a_in_i})=s_i$, where $(s_i-1)v_0\leqslant a_i<s_iv_0$. 
Hence, if 
$s_i':=\min\{s_i,s_0\}$ then 
$\sum _{i}s_i'\geqslant s_0+1$ (use that $r\geqslant 2$). 
By inductive assumption 
all $D_{a_in_i}\in\c L(s_i')_k
\subset \c L^{(v_0)}_k+C_{s_i'}(\c L)_k$ and, therefore, our commutator 
belongs to $\c L^{(v_0)}_k+C_{s_0+1}(\c L)_k$. It remains to note that 
$(\c L^{(v_0)}+C_{s_0+1}(\c L))\cap C_2(\c L)=\c L^{(v_0)}\cap C_2(\c L)+
C_{s_0+1}(\c L)$. 

The lemma is proved.
\end{proof} 

Lemma \ref{L3.5} implies that for 
$a\geqslant (s_0-1)b^*$, all  
$l^{(a)}$ modulo the ideal $\c L^{(v_0)}_k\cap C_2(\c L)_k+C_{s_0+1}(\c L)_k$ 
appear as linear combinations of the linear terms of 
$\gamma *e^{(q)}-e^{(q)}$. 
More precisely, this can be stated as follows. 
\medskip

Let 
$$\sum _{a\in\Z ^+(p)}t^{-aq}(E(at^{b^*})-1)D_{a0}=
\sum _{a',u}\alpha (a',u)t^{-qa'+ub^*}D_{a'0}\, ,$$ 
where 
$a'$ and $u$ run over $\Z ^+(p)$ and 
$\N $, resp., and all $\alpha (a',u)\in \F _p$  
(note that $\alpha (a',1)=a'$).  

\begin{Lem} \label{L3.6}
If $a\geqslant (s_0-1)b^*$ then 
$$l^{(a)}\equiv 
\sum _{\substack{m\geqslant 0 \\ qa'-ub^*=
ap^m}}\alpha (a',u)D_{a',-m}\,\op{mod}\,(\c L^{(v_0)}_k\cap C_2(\c L)_k
+C_{s_0+1}(\c L)_k)\, .$$
\end{Lem}

\begin{proof} [Proof of Lemma] 
Suppose $a_0\in\Z ^0(p)$ satisfies the following inequality 
$a_0\geqslant s_0v_0$. Then  
 $a=qa_0-b^*\geqslant (s_0-1)b^*$ and $l^{(a)}$ is congruent to 
$$a_0D_{a_00}+
\{\text{$k$-(pro)linear combination of $D_{a'm'}$ with $a'>a_0$} \}\, .$$

Since all such $l^{(a)}$ should belong to $\c L^{(v_0)}+C_{s_0+1}(\c L)$,  
this implies that  
all $D_{a_00}$ with $a_0\geqslant s_0v_0$ (or, equivalently, with the weight 
$\geqslant s_0+1$) 
must belong to $\c L^{(v_0)}_k+C_{s_0+1}(\c L)_k$. 
\end{proof} 
Proposition is proved.
\end{proof}

\subsection{Interpretation in $\bar{\c L}^{\dag }$} \label{S3.3}

It remains to prove that in $\bar{\c L}$ we have  
$\bar{\c L}^{(v_0)}+C_{s_0+1}(\bar{\c L})=
\bar{\c L}[v_0]+\bar{\c L}(s_0+1)$. 
By Prop.\ref{P3.4} and Remark from Sect.\ref{S3.2} 
it will be sufficient to establish that 
$$\bar{\c L}^{(v_0)}+\bar{\c L}(s_0+1)\supset \bar{\c L}[v_0]+\bar{\c L}(s_0+1)\, .$$ 
We can use the inductive assumption in the following form 
$$\bar{\c L}[v_0]+\bar{\c L}(s_0)=\bar{\c L}^{(v_0)}+\bar{\c L}(s_0)\, .$$

By the definition  of $\bar{\c L}[v_0]$  and Prop.\ref{P2.15} 
the ideal $\bar{\c L}[v_0]+\bar{\c L}(s_0+1)$ appears as the 
minimal ideal in the set of all ideals 
${\c I}$ of $\bar{\c L}$ such that :
\medskip 

$\bullet $\  ${\c I}\supset \bar{\c D}(s_0+1):=[\bar{\c L}[v_0], 
\bar{\c L}]+\bar{\c L}(s_0+1)=[\bar{\c L}^{(v_0)}, 
\bar{\c L}]+\bar{\c L}(s_0+1)$;
\medskip 

$\bullet $\  if $\iota\in\mathfrak {A}_1^0(p)$  then  
$\c V_kB_k^{\dag }(D^{\dag }_{\iota 0})\in {\c I}_k$.
\medskip

We must prove that 
for any $\iota\in\mathfrak{A}_1^0(p)$, 
$\c V_kB_k^{\dag }(D^{\dag }_{\iota 0})
\in \bar{\c L}_k^{(v_0)}+\bar{\c L}(s_0+1)_k$\ or, equivalently, 
for any $\gamma\ne 0$, 
$$\c V_k(A^{\dag }_{\gamma }-
\id _{\bar{\c L}_k})(D^{\dag }_{\iota 0})\in\bar{\c L}^{(v_0)}_k
+\bar{\c L}(s_0+1)_k\, .$$

Fix $\gamma\ne 0$ and consider equality \eqref{E3.4}. 

By Prop.\ref{P2.13}, 
$(\c VA^{\,\dag }_{\gamma }\otimes\, \id _{\c K})\bar{e}^{\,\dag}
\equiv \bar{e}^{(q)}\,\op{mod}\,(\bar{\c L}[v_0]+\bar{\c L}(s_0))_{\c K}$. 
Hence, there is  ${Z}_{\gamma }
\in (\bar{\c L}[v_0]+\bar{\c L}(s_0))_{sep}$ 
such that 
$$\sigma Z_{\gamma }- Z_{\gamma }=
(\c V(A^{\dag }_{\gamma }-\id _{\bar{\c L}^{\dag }})
\otimes\id _{\c K})\bar e^{\,\dag }\, ,$$
and we obtain (use that $(\bar{\c L}[v_0]+\bar{\c L}(s_0+1))
\,\op{mod}\,\bar{\c D}(s_0+1)$ is abelian)
$$\gamma *\bar{e}^{(q)}=
\sigma (\bar x_{\gamma }\circ Z_{\gamma })\circ 
\bar e^{(q)}\circ (-(\bar x_{\gamma }\circ { Z}_{\gamma }))\,\op{mod}\,
\bar{\c D}(s_0+1)_{sep} .$$

Therefore, the ideal $\bar{\c L}^{(v_0)}+\bar{\c L}(s_0+1)\supset\bar{\c D}(s_0+1)$ 
is the minimal in the family of all ideals ${\c I}$ such that 
\medskip 

$\bullet $\ $\bar{\c L}[v_0]+\bar{\c L}(s_0+1)\supset {\c I}\supset \bar{\c D}(s_0+1)$;
\medskip 

$\bullet $\ $v({Z}_{\gamma }\,\op{mod}\,{\c I}_{sep}/\c K)<qv_0-b^*$ 
(use that $\bar x_{\gamma }$ is defined over $\c K$).
\medskip

By Prop.\ref{P2.15} we have the following congruence modulo 
$\bar{\c D}(s_0+1)_{\c K}$ 
$$
(\c V(A^{\dag }_{\gamma }-\id _{\bar{\c L}^{\dag }})
\otimes\id _{\c K})\bar e^{\dag }\equiv 
\sum _{\op{ch}(\iota )=1}
t^{-\iota }{\c V_k}(A^{\dag }_{\gamma }-
\id _{\bar{\c L}^{\dag }})_kD^{\dag }_{\iota 0}\,.$$

For any $\iota\in\mathfrak{A}_1^0(p)$, consider 
$\bar W_{\gamma \iota }:=
\c V_{k}(A^{\dag }_{\gamma }-\id _{\bar{\c L}^{\dag }})_k
D^{\dag }_{\iota 0}\in\bar{\c L}_k\, .$ 

Recall that $\bar{\c L}[v_0]_k+\bar{\c L}(s_0+1)_k$ 
is generated by all $\bar W_{\gamma \iota }$ 
and the elements of  $\bar{\c D}(s_0+1)_k$. Then 

$$ Z_{\gamma }\equiv \sum \limits _{\op{ch}(\iota )=1} Z_{\gamma \iota }
\,\op{mod}\,\bar{\c D}(s_0+1)_{sep}\, ,$$
where  
$Z_{\gamma \iota }\in (\bar{\c L}[v_0]+\bar{\c L}(s_0+1)/\bar{\c D}(s_0+1))_{sep}$,  
$\sigma  Z_{\gamma \iota }- Z_{\gamma \iota }=t^{-\iota }W_{\gamma \iota }$ 
and $W_{\gamma \iota }:=\bar{W}_{\gamma \iota }\,\op{mod}\,\bar{\c D}(s_0+1)_k
\in ((\bar{\c L}[v_0]+\bar{\c L}(s_0+1)/\bar{\c D}(s_0+1))_{k}$. 
\medskip 

All above $Z_{\gamma \iota }$ come from elementary Artin-Schreier equations. 
Indeed, suppose $\{\omega _j\}$ is a (finite) $\F _p$-basis of 
$(\bar{\c L}[v_0]+\bar{\c L}(s_0+1))/\bar{\c D}(s_0+1)$. Then for some 
$w_{\gamma\iota j}\in k$, 
$W_{\gamma \iota }=\sum \limits _jw_{\gamma \iota j}\omega _j$ 
and $Z_{\gamma\iota }=\sum\limits _jz_{\gamma \iota j}\omega _j$, where 
$z_{\gamma \iota j}^p-z_{\gamma \iota j}=w_{\gamma \iota j}t^{-\iota }$. 
In particular, for any fixed $\iota $ (and $\gamma $), 
$\c K(Z_{\gamma \iota })$ is a composit of all $\c K(z_{\gamma \iota j})$. Therefore, 
$\c K(Z_{\gamma \iota })/\c K$ is an elementary abelian $p$-extension, 
which is either trivial or 
has only one (upper) 
ramification number $\iota p^{-v_p(\iota )}$.
\medskip 

This implies that  
\medskip 

--- $\c K( Z_{\gamma }\,\op{mod}\,{\c I}_{sep})/\c K$ 
coincides with the composite of  all 
\linebreak 
$\c K(Z_{\gamma\iota }\,\op{mod}\,({\c I}/\bar{\c D}(s_0+1))_{sep})/\c K$ 
(use that for different $\iota $ these extensions are linearly disjoint 
because  by Prop.\,\ref{P2.8}a) their ramification numbers are different).
\medskip 

In particular, 
\medskip 

--- if $W_{\gamma\iota }\notin \c I_k$ then the field 
$\c K(Z_{\gamma\iota }\,\op{mod}\,(\bar{\c I}/\bar{\c D}(s_0+1))_{sep})/\c K$ 
is a finite abelian $p$-extension with only one ramification number 
$\iota p^{-v_p(\iota )}$;
\medskip 

--- by Prop.\ref{P2.8}a), the ramification numbers of different non-trivial extensions 
$\c K(Z_{\gamma\iota }\,\op{mod}\,({\c I}/\bar{\c L}^*(s_0+1))_{sep})/\c K$ 
are different.  
\medskip 

As a result, the biggest upper ramification number of the field extension 
$\c K(Z_{\gamma }\,\op{mod}\,\c I_{sep})/\c K$ 
coincides with  $\max\{\iota p^{-v_p(\iota )}\ |\ W_{\gamma\iota }\notin \c I_k\}$. 

By Prop.\,\ref{P2.8}b), if $\iota\in\mathfrak{A}_1^0(p)$ then 
$\iota p^{-v_p(\iota )}\geqslant qv_0-b^*$. 
This implies that the biggest upper ramification number 
$v(\c K(Z_{\gamma }\,\op{mod}\,{\c I}_{sep})/\c K)<qv_0-b^*$ if and only if 
all $W_{\gamma\iota }\in\c I_k$, i.e. $\c I=
\bar{\c L}[v_0]+\bar{\c L}(s_0+1)$.  

Theorem \ref{T3.1} is completely proved. 
\medskip 

\section{Construction of explicit generators of $\bar{\c L}{[v_0]}$} \label{S4} 

\subsection{Choice of $e\in\c L_{\c K}$} \label{S4.1}  
In 
\cite{Ab1,Ab2,Ab3} we fixed the group 
isomorphism $\c G_{<p}\simeq G(\c L)$ 
induced by the epimorphism 
$\eta _e=\pi _f(e):\c G\To G(\c L)$ via 
a special choice of $e\in\c L_{\c K}$. In this paper we 
use more general element $e$ by assuming that 
\begin{equation} \label{E4.1} \wt{\exp}\,e\equiv 1+\sum _{1\leqslant s<p}
\eta (a_1,\dots ,a_s)t^{-(a_1+\dots +a_s)}D_{a_10}\dots D_{a_s0}
\,\op{mod}\,\c J^p_{\c K}\, .
\end{equation} 
Here $\c J$ is the augmentation ideal in the enveloping algebra 
$\c A$ of $\c L$. 
 In the above sum the indices $a_1,\dots ,a_s$ run 
over $\Z ^0(p)$ and  the ``structural constants'' 
$\eta (a_1,\dots ,a_s)\in k$  
satisfy the following 
identities:  
\medskip 

${\rm I_e}$)\   $\eta (a_1)=1$;
\medskip 

${\rm II}_e$)\  if $0\leqslant s_1\leqslant s<p$ then 
$$\eta (a_1,\dots ,a_{s_1})\eta (a_{s_1+1},\dots ,a_s)=
\sum _{\pi\in I_{s_1s}} \eta (a_{\pi (1)},\dots ,a_{\pi (s)}),$$
where $I_{s_1s}$
consists of all permutations $\pi $ of order $s$ such that the sequences 
$\pi ^{-1}(1),\dots ,\pi ^{-1}(s_1)$ and
$\pi ^{-1}(s_1+1),\dots ,\pi ^{-1}(s)$
are increasing 
(i.e. $I_{s_1s}$ is the set of all \lq\lq insertions\rq\rq\
of the ordered set $\{1,\dots ,s_1\}$ into
the ordered set $\{s_1+1,\dots ,s\}$). 

Assumption ${\rm I_e}$) means that $e$ satisfies $\eqref{E1.1}$ from Sect.\ref{S1}. 

Assumption ${\rm II}_e$) means that 
$$\Delta (\wt{\exp}(e))\equiv \wt{\exp}(e)\otimes\wt{\exp}(e)\,\op{mod}\,
(\c J_{\c K}\hat\otimes 1+1\hat\otimes \c J_{\c K})^p\, ,$$ 
i.e. 
$\wt{\exp}(e)$ is diagonal modulo degree $p$. This means that 
$e$ is a $k$-linear combination of the commutators 
$t^{-(a_1+\dots +a_r)}[\ldots [D_{a_10},\ldots ],D_{a_s0}]$. 
In particular, $e$ satisfies 
the assumption \eqref{E1.3} from Sect.\ref{S1} and the compatibility 
\eqref{E2.5} can be easily satisfied. Therefore, 
we can use Theorem \ref{T3.1} 
to obtain generators of the ramification ideal $\c L^{(v_0)}$.
Note that in most applications  
of the results from \cite{Ab1, Ab2, Ab3} we 
used  
the simplest choice $e=\sum _{a\in\Z ^0(p)}t^{-a}D_{a0}$, where all 
$\eta (a_1,\dots ,a_s)=1/s!$ 

\subsection{Statement of the main result}\label{S4.2} 

For $\bar a=(a_1,\dots ,a_s)$ with all $a_i\in\Z ^0(p)$, set 
$\eta (\bar a)=\eta (a_1,\dots ,a_s)$. 

\begin{definition} 
 Let $\bar n=(n_1,\dots ,n_s)$ with $s\geqslant 1$. Suppose there is a partition  
 $0=i_0<i_1<\dots <i_r=s$ such that 
 if $i_j<u\leqslant i_{j+1}$ then 
 $n_u=m_{j+1}$ and $m_1>m_2>\dots >m_r$. Then  set  
 $$\eta (\bar a,\bar n)_s=
 \sigma ^{m_1}\eta (\bar a^{(1)})\dots \sigma ^{m_r}\eta (\bar a^{(r)})\, ,$$
 where $\bar a^{(j)}=(a_{i_{j-1}+1}, \dots ,a_{i_j})$. 
 If such a partition does not exist we set $\eta (\bar a,\bar n)_s=0$. 
 (If there is no risk of confusion we just write  $\eta (\bar a,\bar n)$ 
 instead of $\eta (\bar a,\bar n)_s$.)
\end{definition}

If $s=0$ we set $\eta (\bar a,\bar n)_s=1$. 

For $\bar a=(a_1,\dots ,a_s)$, $\bar n=(n_1,\dots ,n_s)$,   
set $D_{(\bar a,\bar n)}=D_{a_1n_1}\dots D_{a_sn_s}$.

Note, if  $e_{(N^*,\,0]}:=\sigma ^{N^*-1}(e)\circ 
\sigma ^{N^*-2}(e)\circ \dots \sigma (e)\circ e$ then 
$$\wt{\exp}\,e_{(N^*,\,0]}\equiv \sum _{\bar a, \bar n}\eta (\bar a,\bar n)_s
D_{(\bar a, \bar n)}\, \op{mod}\,\c J_{\c K}^p.$$

For $\alpha\geqslant 0$ and $N\in\Z _{\geqslant 0} $, introduce 
$\c F^0_{\alpha ,-N}\in{\c L}_k$ such that 
$$\c F^0_{\alpha ,-N}=\sum _{\substack {1\leqslant s<p\\
a_i, n_i }}a_1\eta (\bar a,\bar n)[\dots [D_{a_1 n_1},
D_{a_2 n_2}],\dots ,D_{a_s n_s}]\, .$$

Here: 

--- \ $\bar a=(a_1,\dots ,a_s)$, $n_1=0$ and all $n_i\geqslant -N$;

--- \ $\alpha =\gamma (\bar a,\bar n)=a_1p^{n_1}+a_2p^{n_2}+\dots +a_sp^{n_s}$\,.
\medskip 

Note that non-zero terms in the above expression for 
$\c F_{\alpha , -N}^0$ can appear only if 
$0=n_1\geqslant n_2\geqslant\ldots \geqslant n_s$ and $\alpha\in A[p-1,N]$.  

Our result about explicit generators of $\bar{\c L}[v_0]$ can be stated 
in the following form.

Let $\bar{\c F}^0_{\alpha ,-N}$ be the image of $\c F^0_{\alpha ,-N}$ in 
$\bar{\c L}_k$.

If $\iota =qp^m\alpha -p^mb^*\in\mathfrak{A}^0_1$ is the standard presentation 
from Sect.\ref{S2.2} we indicate the dependance of $\alpha $ and $m$ on 
$\iota $ by setting $\alpha =\alpha [\iota ]$ and $m=m[\iota ]$. 
Recall that $\alpha [\iota ]\in A[p-1,m[\iota ]]$ and $m[\iota ]<N^*$. 

Let $m(\iota )$ be the maximal non-negative integer such that 
$\iota p^{m(\iota )}\leqslant (p-1)b^*$. 
For any $\iota\in\mathfrak{A}^0_1$, fix a choice of $m_{\iota }\geqslant m(\iota )$.

\begin{Thm} \label{T4.1} 
 $\bar{\c L}[v_0]$ is the minimal ideal in $\bar{\c L}$ such that 
for all $\iota\in\mathfrak{A}^0_1$ with $\alpha [\iota ]\geqslant 0$, 
$\bar{\c F}^0_{\alpha [\iota ], -(m[\iota ]+m_{\iota })}\in\bar{\c L}[v_0]_k$.
\end{Thm}

The proof is given in Sect.\,\ref{S4.3}-\ref{S4.6} below.

\subsection{Recurrent relation} \label{S4.3} 

We are going to carry out computations in the enveloping algebra 
$\bar{\c A}$ of the Lie algebra 
$\bar{\c L}$. 
 Note that the natural embedding  
$\bar{\c L}_{\c K}\subset \bar{\c A}_{\c K}$ remains still injective when taken modulo 
$\bar{\c J}_{\c K}^p$. This can be established similarly to the corresponding property 
for Lie $\F _p$-algebras from Sect.\ref{S1.2}. 

Using universal properties of 
enveloping algebras obtain the following lemma. 
(We are going to use these properties slightly later.) 

\begin{Lem} \label{L4.2}
Suppose $I$ is an ideal in the Lie algebra 
$L$ of nilpotence class $<p$. Let $A$ be an enveloping algebra of $L$ 
with augmentation ideal $J$ 
and $J_I:=IA$ -- the corresponding 
(two-sided) ideal in $A$. Then:
\medskip 

{\rm a)}\ $(J_I+J^p)\cap L=I$;
\medskip 

{\rm b)}\ $(JJ_I+J_IJ+J^p)\cap L=[I,L]$.
\end{Lem}

 Consider relation \eqref{E2.5} and choose $\wt{e}^{\,sp}=
\sum _{\iota }t^{-\iota }l_{\iota }^{sp}$ such that 
for all $\iota\in\mathfrak{A}^0(p)$ with $\op{ch}(\iota )\geqslant 1$, 
$l_{\iota }^{sp}=D^{\dag }_{\iota 0}$ if $\iota\in\mathfrak{A}^0(p)$ 
and $l_{\iota }^{sp}=0$, otherwise. 
In other words, the part of $\wt{e}^{\,sp}$ which \lq\lq disappears 
under $\c V_{\c K}$\rq\rq\  
coincides with $\sum\limits _{\op{ch}(\iota )
\geqslant 1 }t^{-\iota }D_{\iota 0}^{\dag }$. 

Note that 
$\wt{\exp}\,(U*\bar{e}^{\,(q)})\equiv \wt{\exp}\,\bar{e}^{\,(q)}+
\bar{\c E}U\,\op{mod}\,  
\bar{\c A}_{\c K}U^2$, 
where 
$$\bar{\c E}=\hskip -5pt 
\sum _{\substack{s\geqslant 1\\a_i\in \Z ^0(p)}}\hskip -5pt
\eta (a_1,\dots ,a_s)
t^{-(a_1+\dots +a_s)q+b^*}(a_1+\dots +a_s)D_{a_10}\dots 
D_{a_s0}\, .$$

Apply $\wt{\exp }$ to \eqref{E2.7} and find a 
lift $\bar x$ of $\tilde x$ to $\bar{\c N}^{(q)}$ 
such that  
$$\wt{\exp }\,\bar e^{(q)}+\bar{\c E}U\equiv  
 (1+U\sigma\bar x)
\left (\wt{\exp}\,\bar e^{(q)}+U
\sum\limits _{\op{ch}(\iota )\geqslant 1}
t^{-\iota }\c V_kB^{\dag }_k(D^{\dag }_{\iota 0})\right )(1-U\bar x)$$
modulo $\bar{\c J}_{\c K}^pU+\bar{\c A}_{\c K}U^2$ (proceed similarly 
to the proof of Prop.\ref{P1.4}). 
Comparing the coefficients for $U$   
and setting $\c V_kB^{\dag }_k(D^{\dag }_{\iota 0})=V_{\iota 0}$ 
we obtain in $\bar{\c A}_{\c K}$ the following 
congruence modulo $\bar{\c J}_{\c K}^{\,p}$ 
\begin{equation} \label{E4.2} 
\sigma \bar x-\bar x+\sum\limits _{\iota }t^{-\iota }V_{\iota 0}\equiv 
\bar{\c E}+(\wt{\exp}\,\bar e^{\,(q)}-1)\cdot \bar x-
\sigma \bar x\cdot (\wt{\exp}\,\bar e^{\,(q)}-1)\, .
\end{equation} 
This equality gives a recurrent procedure to determine uniquely the elements 
$\bar x\in \sum\limits _{\op{ch}(\iota )\geqslant 1}
t^{-\iota }\bar{\c L}_k+t^{(p-1)b^*}\bar{\c N}^{(q)}$ and $V_{\iota 0}\in\bar{\c L}_k$. 
\medskip

\subsection{Some combinatorial identities} \label{S4.4}

Let 
$$-e_{[\,0,\,N^*)}:=(-e)\circ (-\sigma e)\circ\ldots\circ  (-\sigma ^{N^*-1}e)$$ 
and introduce 
the constants $\eta ^o(\bar a,\bar n)\in k$ by the following congruence 
$$\wt{\exp}(-e_{[\,0,\,N^*)})\equiv \sum \eta ^o(\bar a,\bar n)D_{(\bar a,\bar n)}\,
\op{mod}\,\c J_{\c K}^p.$$

Set $\eta ^o(\bar a):=\eta ^o(\bar a,\bar 0)$.

It can be easily seen that if 
there is a partition from the  
definition of $\eta $-constants in Sect.\ref{S4.2} such that  $m_1<m_2<\dots <m_r$ then 
$$\eta ^o(\bar a, \bar n)_s=\sigma ^{m_1}
\eta ^o(\bar a^{(1)})\cdot \sigma ^{m_2}
\eta ^o(\bar a^{(2)})\cdot 
\ldots \cdot \sigma ^{m_r}\eta ^o(\bar a^{(r)})\, .$$
Otherwise,  $\eta ^o(\bar a,\bar n)_s=0$.  

If there is no risk of confusion we just write  $\eta (1,\dots ,s)$ instead of 
$\eta (\bar a,\bar n)_s$ and use 
the similar agreement for $\eta ^o$. E.g. the equalities 
$$e_{(N^*,\,0]}\circ (-e_{[\,0,\,N^*)})=e_{[\,0,\,N^*)}\circ (-e_{(\,N^*,\,0]})=0$$ 
can be written as the following identities
\begin{equation} \label{E4.3} 
\sum _{0\leqslant s_1\leqslant s}\eta (1,\dots ,s_1)\eta ^o(s_1+1,\dots ,s)=
\qquad\qquad\qquad\qquad\qquad 
\end{equation} 
$$\qquad\qquad\qquad\qquad 
\sum _{0\leqslant s_1\leqslant s}\eta ^o(1,\dots ,s_1)
\eta (s_1+1,\dots ,s)=\delta _{0s} 
$$
(here $\delta _{0s}$ is the Kronecker symbol). 
\medskip



 For $1\leqslant  s_1\leqslant s<p$, 
consider the subset $\Phi _{ss_1}$ of
permutations $\pi $ of order $s$ such that $\pi (1)=s_1$
and for any $1\leqslant l\leqslant s$,
the subset $\{\pi (1),\dots ,\pi (l)\}$ of the
segment $[1,s]$ is \lq\lq connected\rq\rq ,
i.e. there exists $n(l)\in\N$ such that
$$\{\pi (1),\dots ,\pi (l)\}=\{n(l),n(l)+1,\dots ,n(l)+l-1\}.$$
By definition, we set $\Phi _{s0}=\Phi _{s,s+1}=\emptyset $.
\medskip   

Set $B_{s_1}(1,\dots ,s)=\sum_{\pi\in\Phi _{ss_1} }
\eta (\pi (1),\dots ,\pi (s))$. 
\medskip 

Note that: 

--- $B_0(1,\dots ,s)=B_{s+1}(1,\dots ,s)=0$; 

--- $B_1(1,\dots ,s)=\eta (1,2,\dots ,s)$; 

--- $B_s(1,\dots ,s)=\eta (s,s-1,\dots ,1)$.

\begin{Lem} \label{L4.3} Suppose $0\leqslant s_1\leqslant s<p$. Then: 

{\rm a)}\ $B_{s_1}(1,\dots ,s)+B_{s_1+1}(1,\dots ,s)=\eta (s_1,\dots ,1)
\eta (s_1+1,\dots ,s)$; 

{\rm b)}\ $\eta ^o(1,\dots ,s)=(-1)^s\eta (s,s-1,\dots ,1)$;

{\rm c)}\ for indeterminates $X_1,\dots ,X_s$, 
$$\sum_{\substack{1\leqslant s_1\leqslant s \\ \pi\in\Phi _{ss_1}}} 
(-1)^{s_1-1}X_{\pi ^{-1}(1)}\dots X_{\pi ^{-1}(s)}=
[\dots [X_1,X_2],\dots ,X_s].$$
\end{Lem}

\begin{proof} 
 a) Use that all insertions of 
 $(s_1,...,1)$ into 
$(s_1+1,\dots s)$ are  
 \lq\lq connected\rq\rq\  and start either with $s_1$ or $s_1+1$.
 
 b) Clearly, part a) implies that 
 $$\sum _{0\leqslant s_1\leqslant s}(-1)^{s_1}
 \eta (s_1,\dots ,1)\eta (s_1+1,\dots ,s)=\delta _{0s}\, .$$
 Then our statement follows from above relation \eqref{E4.3}. 
 
 c) Use that the right-hand side is a linear 
 combintion of the monomials $X_{i_1}\dots X_{i_s}$ such that 
 for any $l\geqslant 1$, $\{j\ |\ i_j\in [1,l]\}$ is 
 a ``connected'' segment of consecutive $l$ integers. 
\end{proof}
 
 \subsection{Lie elements $\bar{\c F}[\iota ]$ and $\bar{\c F}[\iota ]_0$} \label{S4.5} 

Introduce the following notation:
\medskip 

--- \ $\bar n=(n_1,\dots ,n_s)\geqslant M$ means that all  
$n_i\geqslant M$. Similarly, we interpret $\bar n>M$, $\bar n\leqslant M$ and $\bar n<M$. 
\medskip

--- \ $\gamma (\bar a,\bar n)=a_1p^{n_1}+\dots +a_sp^{n_s}$. 
\medskip

 For $1\leqslant s_1\leqslant s$, let 
$\gamma ^*_{[s_1,s]}(\bar a,\bar n)=\sum\limits _{s_1\leqslant u\leqslant s}a_up^{n_u}$  
where $n_u^*=0$ if $n_u=M(\bar n):=\max\{n_1,\dots ,n_s\}$ and 
$n_u^*=-\infty $  (i.e. $p^{n_u^*}=0$), otherwise. 
\medskip 

 For $\iota\in\mathfrak{A}^0_1$, introduce 
$$\bar{\c F}[\iota ]=\sum _{(\bar a,\bar n)}\sum _{1\leqslant s_1\leqslant s}
 \eta ^o(1,\dots ,s_1-1)\eta (s_1,\dots ,s)
 \gamma _{[s_1,s]}^{*}(\bar a,\bar n)D_{(\bar a,\bar n)}\in\bar{\c A}_k\, .$$ 
Here the first sum is taken over 
 all $(\bar a,\bar n)$ of lengths $1\leqslant s<p$ such that $\bar n\geqslant 0$ and  
 $\gamma (\bar a,\bar n)-p^{M(\bar n)}b^*=\iota $. Note that 
 $M(\bar n)$ depends only on $\iota $ and, therefore, all non-zero summands 
 in $\bar{\c F}[\iota ]$ depend on $(\bar a,\bar n)$ with the same $M(\bar n)$. 
 \medskip

 Let $\bar{\c F}[\iota ]_0$ be a part of the above sum taken under 
the condition $m(\bar n):=\min\{n_1,\dots ,n_s\}=0$. 
Then for any $\iota\in\mathfrak{A}^0_1$ and $m\geqslant 0$,  
$$\sigma ^m\bar{\c F}[\iota ]_0+\sigma ^{m-1}
\bar{\c F}[\iota p]_0+\dots +\bar{\c F}[\iota p^m]_0=\bar{\c F}[\iota p^m]\, .$$
In particular, $\bar{\c F}[\iota ]=
\sum _{\iota ', m}\sigma ^m\bar{\c F}[\iota ']_0$ where the sum is taken over 
all $\iota '\in\mathfrak {A}^0_1$ and $m\geqslant 0$ such that 
$\iota 'p^m=\iota $. 

\begin{Prop} \label{P4.4}
If $\iota =qp^m\alpha -p^mb^*\in\mathfrak{A}^0_1$ (standard notation) 
then 
$\bar{\c F}[\iota p^n]=\sigma ^{m+n}\bar{\c F}^0_{\alpha , -(m+n)}$.
\end{Prop}

\begin{proof} We have 
$$\sigma ^{-(m+n)}\bar{\c F}[\iota p^n]=
\sum _{\substack{1\leqslant s_1\leqslant s<p\\ (\bar a,\bar n)}}
\eta ^o(1,\dots ,s_1-1)\eta (s_1,\dots ,s)
\gamma _{[s_1,s]}^{*}(\bar a,\bar n)D_{(\bar a,\bar n)}\, ,$$ 
where the sum is taken for $(\bar a,\bar n)$ with $M(\bar n)=0$, 
$\bar n\geqslant -(m+n)$  
and $\gamma (\bar a,\bar n)=\alpha $. By Lemma \ref{L4.3},   
$\eta ^o(1,\dots ,s_1-1)=(-1)^{s_1-1}\eta (s_1-1,\dots ,1)$,  
and we obtain 
$$\sum _{\substack{1\leqslant s_1\leqslant s<p \\ (\bar a,\bar n) }}
(-1)^{s_1-1}(B_{s_1-1}(1,\dots ,s)+
B_{s_1}(1,\dots ,s))\gamma ^*_{[s_1,s]}(\bar a,\bar n)D_{(\bar a\bar n)}=$$

$$\sum _{\substack{1\leqslant s_1\leqslant s<p \\ (\bar a,\bar n)}}
(-1)^{s_1-1}B_{s_1}(1,\dots ,s)(\gamma ^*_{[s_1,s]}(\bar a,\bar n)-
\gamma ^*_{[s_1+1,s]}(\bar a,\bar n)){D}_{(\bar a,\bar n)}=$$

$$\sum_{\substack{1\leqslant s<p \\ (\bar a,\bar n)}}
\sum_{1\leqslant s_1\leqslant s} (-1)^{s_1-1}
B_{s_1}(1,\dots ,s)a_{s_1}p^{n_{s_1}^*}{D}_{(\bar a\bar n)}=$$

$$\sum_{\substack{1\leqslant s<p \\ (\bar a,\bar n)}}
\sum_{\substack{1\leqslant s_1\leqslant s\\ \pi\in\Phi _{ss_1}}}
(-1)^{s_1-1}\eta ({\pi (1)},\dots ,{\pi (s)})
a_{s_1}p^{n_{s_1}^*}D_{(\bar a,\bar n)}=$$

$$ \sum_{\substack{1\leqslant s<p \\ (\bar a,\bar n)}}
\eta (1,\dots ,s)a_1
\sum_{\substack{1\leqslant s_1\leqslant s\\ \pi\in\Phi _{ss_1} }}
(-1)^{s_1-1}D_{a_{\pi ^{-1}(1)}n_{\pi ^{-1}(1)}}\dots
D_{a_{\pi ^{-1}(s)}n_{\pi ^{-1}(s)}} =$$
$$\sum _{\substack{1\leqslant s<p\\(\bar a,\bar n)}}
\eta (1,\dots ,s)a_1[\dots [D_{a_1n_1},D_{a_2n_2}],\dots ,D_{a_sn_s}]
=\bar{\c F}^0_{\alpha ,-(m+n)}\,.$$
The proposition is proved. 
\end{proof} 

\begin{Cor} \label{C4.5}
 All $\bar{\c F}[\iota p^m]$ and $\bar{\c F}[\iota p^m]_0$ belong to $\bar{\c L}_k$.
\end{Cor}

\subsection{Solving recurrent relation \eqref{E4.2}} \label{S4.6} 
For $\iota\in\mathfrak{A}^0$, let 
\medskip

--- \ $m(\iota ):=\max\{m\ |\ \iota p^m \in\mathfrak{A}^0\}
(=\max\{m\ |\ |\iota p^m|\leqslant (p-1)b^*\})$. 
\medskip 

--- \ $\mathfrak{A}_1^{\op{prim}}=\mathfrak{A}^0_1\setminus p\mathfrak{A}^0_1$ 
(note that $\mathfrak{A}^0_1(p)=\{\iota\in\mathfrak{A}^{\op{prim}}_1\ |\ \iota >0\}$). 
\medskip 

As earlier, set $\bar{\c D}:=[\bar{\c L}[v_0],\bar{\c L}]$, 
$\bar{\c L}[v_0](s):=\bar{\c L}[v_0]+\bar{\c L}(s)$ 
and $\bar{\c D}(s):=\bar{\c D}+\bar{\c L}(s)$. Clearly, 
$\bar{\c L}[v_0]=\bar{\c L}[v_0](p)$ and 
$\bar{\c D}=\bar{\c D}(p)$.

\begin{Prop} \label{P4.6}  
{\rm a)}\ $\bar x\equiv \sum \limits _{\iota ,m}
 \bar{\c F}[\iota p^m]
 t^{-\iota p^m}\,\op{mod}\,\bar{\c L}[v_0]_{\c K}$, 
 where the sum is taken over all $i\in\mathfrak{A}^{\op{prim}}_1$ 
 and $m\geqslant 0$; 
 \medskip 
 
 {\rm b)}\ if $\iota\in\mathfrak{A}_1^0(p)$, 
 then 
 $V_{\iota 0}\equiv -
 \sigma ^{-m(\iota )}\bar{\c F}[\iota p^{m(\iota )}]\,
 \op{mod}\,\bar{\c D}_k\, .$
\end{Prop} 

\begin{proof} Apply  induction on $1\leqslant s_0< p$ 
by assuming that a) holds modulo $\bar{\c L}[v_0](s_0)_{\c K}$
and deducing from 
this that a) and b) hold modulo the ideals $\bar{\c L}[v_0](s_0+1)_{\c K}$ and, 
resp., $\bar{\c D}(s_0+1)_k$.  

Clearly, a) holds modulo $\bar{\c L}[v_0](1)_{\c K}=\bar{\c L}_{\c K}$. 

Suppose $1\leqslant s_0<p$ and part a) holds modulo 
$\bar{\c L}[v_0](s_0)_{\c K}$. Applying this assumption to the 
right-hand side of \eqref{E4.2} we obtain (use \eqref{E4.3}) 
that 
\begin{equation} \label{E4.4} 
\sigma\bar x-\bar x+\sum _{\iota }t^{-\iota }V_{\iota 0}
\equiv -\sum _{\iota , m}\bar{\c F}[\iota p^m]_0t^{-\iota p^m}
\end{equation} 
modulo ($\bar{\c J}\bar{\c J}_{\bar{\c L}[v_0](s_0)}+
\bar{\c J}_{\bar{\c L}[v_0](s_0)}\bar{\c J}
+\bar{\c J}^p)_{\c K}$, cf. notation from Lemma \ref{L4.2}. 
(Here the right-hand sum 
is taken over all 
$\iota\in\mathfrak{A}^{\op{prim}}_1$ and $m\geqslant 0$.) 

Since the both parts of congruence \eqref{E4.4} belong to 
$\bar{\c L}_{\c K}$, part b) of Lemma \ref{L4.2} 
implies that \eqref{E4.4} holds modulo 
$[\bar{\c L}[v_0](s_0),\bar{\c L}]_{\c K}=\bar{\c D}(s_0+1)_{\c K}$. 

\begin{remark} 
Since $\bar x, \sigma \bar x\in\bar{\c N}^{(q)}$ relation 
\eqref{E4.4} implies that  
$\bar{\c F}[\iota p^m]_0\in \bar{\c D}(s_0+1)_k=\bar{\c D}_k+\bar{\c L}(s_0+1)_k$  
if $\iota p^m>s_0b^*$. 
\end{remark}

Apply the operators $\c S$ and $\c R$ from Lemma \ref{L2.11} to recover the elements 
$\sum _{i\in\mathfrak{A}^0(p)}t^{-\iota }V_{\iota 0}$ 
and $\bar x$ modulo $\bar{\c D}(s_0+1)_{\c K}$ as follows.

Let $\bar x=\bar x^++\bar x^-$, where $\bar x^+$ (resp., $\bar x^-$) 
is the linear combination of elements of $\bar{\c L}_k$ with positive 
(resp., negative) powers of $t$. 

If $\iota <0$ then $\c S(\bar{\c F}[\iota p^m]_0t^{-\iota p^m})=
-\sum \limits _{n\geqslant 0}\sigma ^n\bar{\c F}[p^m\iota ]_0t^{-\iota p^{n+m}}$ 
and, therefore, 
$\bar x^+\equiv \sum\limits  _{\iota ,m}
\bar{\c F}[\iota p^m]t^{-\iota p^m}\,\op{mod}\,\bar{\c D}(s_0+1)_{\c K}$, 
where the sum is taken over all 
$\iota \in\mathfrak{A}^{\op{prim}}_1\setminus \mathfrak{A}^0_1(p)$ and $m\geqslant 0$. 
This gives part a) modulo $\bar{\c L}[v_0](s_0+1)_{\c K}\supset \bar{\c D}(s_0+1)_{\c K}$ 
at the level of positive powers of $t$. 
\medskip

Let $\iota \in\mathfrak{A}^0_1(p)$. Then 
$\c R(t^{-\iota p^m}\bar{\c F}[\iota p^m]_0)=
t^{-\iota }\sigma ^{-m}\bar{\c F}[\iota p^m]_0$ and 

$$V_{\iota 0}t^{-\iota }\equiv 
-t^{-\iota }
\sum _{0\leqslant m\leqslant m(\iota )}\sigma ^{-m}\bar{\c F}[\iota p^m]_0
\equiv 
-t^{-\iota }\sigma ^{-m(\iota )}\bar{\c F}[\iota p^{m(\iota )}]\, $$
modulo $\bar{\c D}(s_0+1)_{\c K}$. 
This gives part b).

As a result,  we have the following congruences modulo  
$\bar{\c L}[v_0](s_0+1)_{\c K}$: 

$$\c S(\bar x^-)\equiv -\sum _{\iota ,m}
\c S(t^{-\iota p^m}\bar{\c F}[\iota p^m]_0)
\equiv -\sum _{\iota , m}
\sum _{0\leqslant m_1<m}t^{-\iota p^{m_1}}
\sigma ^{-m+m_1}\bar{\c F}[\iota p^m]_0
$$
$$\equiv -\sum _{\iota , m}t^{-\iota p^{m}}
\sum _{m_1>m}
\sigma ^{-m_1+m}\bar{\c F}[\iota p^{m_1}]_0
\equiv -\sum _{\iota , m}t^{-\iota p^{m}}
\sum _{\substack{m_1+m_2=m \\ m_2<0}}
\sigma ^{m_2}\bar{\c F}[\iota p^{m_1}]_0$$
$$\equiv 
\sum _{\iota , m}t^{-\iota p^{m}}
\sum _{\substack{m_1+m_2=m \\m_1, m_2\geqslant 0}} 
\sigma ^{m_2}\bar{\c F}[\iota p^{m_1}]_0=
\sum _{\iota ,m}\bar{\c F}[\iota p^m]t^{-\iota p^m}$$
(here 
$\iota $ and $m$ run over $\mathfrak{A}^0_1(p)$ and, resp., $\Z _{\geqslant 0}$) because 
$$\sum _{m_2+m_1=m}
\sigma ^{m_2}\bar{\c F}[\iota p^{m_1}]_0=
\bar{\c F}[\iota p^m]\equiv 
\sigma ^{m-m(\iota )}V_{\iota }
\,\op{mod}\, \bar{\c L}[v_0](s_0+1)_k\, .$$
This completes the induction step for part a). 
\end{proof} 

\begin{Cor} \label{C4.7} 
$\bar{\c L}[v_0]$ is the minimal ideal 
in $\bar{\c L}$ such that $\bar{\c L}[v_0]_k$ contains 
 all $\bar{\c F}[\iota p^{m_{\iota}}]$. 
\end{Cor}
\begin{proof} 
If $m>m(\iota )$ then $\iota p^m>(p-1)b^*$ and 
by remark in the proof of Prop.\,\ref{P4.6}, 
$\bar{\c F}[\iota p^m]_0\in \bar{\c D}_k$ . 
Therefore, 
$\bar{\c F}[\iota p^{m_{\iota }}]\equiv \bar{\c F}
[\iota p^{m(\iota )}]\,\op{mod}\,\bar{\c D}_k$. 
\end{proof}
Theorem \ref{T4.1} is completely proved.

\section{Effective choice of parameters} \label{S5} 

Theorem \ref{T4.1} gives explicit description of the ramification 
ideal $\c L^{(v_0)}$ 
but this description depends on a choice of parameters 
$\delta _0$, $b^*$ and $q=p^{N^*}$ involved into the construction  of the module of 
auxilliary coefficients $\mathfrak{A}^0$. More precisely, the corresponding generating elements 
${\c F}^0_{\alpha [\iota ], -(m[\iota ]+m(\iota ))}$ of the ideal 
$\c L^{(v_0)}_k$ depend on 
the rational numbers $\alpha [\iota ]\geqslant v_0$ and the integers 
$m[\iota ]$ coming from the appropriate 
$\iota\in\mathfrak{A}^0_1(p)$. The values of 
$\delta _0$, $b^*$ and $N^*$ 
can be specified in a sufficiently constructive way directly from their definitions 
but it is highly unlikely that this could be done in a more or less optimal way. The reason is that a choice of $\delta _0$, $b^*$ and $N^*$ 
depends on the whole set 
$\mathfrak{A}^0$ but the construction of 
generators uses only the subset $\mathfrak{A}^0_1(p)\subset\mathfrak{A}^0$.

As we have noticed in the Introduction, an analogue of Theor.\,\ref{T4.1} was obtained 
in \cite{Ab1} by different methods, and was stated in the terms of generators 
${\c F}^0_{\alpha ,-N}$ with arbitrary rationals $\alpha\geqslant v_0$ and 
a boundary value $\wt{N}(v_0)$ such that 
$N\geqslant \wt{N}(v_0)$.   
In Sect.\,\ref{S5.1} we deduce this analogue from Theor.\,\ref{T4.1} 
and prove that it holds with the boundary value 
$\wt{N}(v_0)=N^*-1$. 

It should be noticed that if $\wt{N}(v_0)$ is (unreasonably) 
growing  then the number of 
dependent generators among 
$\c F^0_{\alpha ,-N}$, $\alpha\geqslant v_0$, also grows. As a result,   
the description of the ideal $\c L^{(v_0)}$ is getting more 
and more complicated. 
In Sect.\,\ref{S5.2} we use the left-continuity property 
of ramification filtration to provide \lq\lq flexible\rq\rq\  boundaries 
$\wt{N}(v_0, \alpha )$ depending on the parameter 
$\alpha $. This allows us to obtain more effective description 
of the whole filtration 
$\{\c L^{(v)}\}_{v\geqslant 1}$ under the condition that the set of its jumps is known. 

\subsection{Relation to the main result of \cite{Ab1}} \label{S5.1} \ \ 

In \cite{Ab1} (cf. also \cite{Ab3})  we proved the following theorem. 

\begin{Thm} \label{T5.1} 
There is $\wt N(v_0)\in\N $ such that if $N\geqslant  \wt N(v_0)$ 
is fixed then $\c L^{(v_0)}$ is the minimal ideal 
in $\c L$ such that for all $\alpha\geqslant v_0$, 
$\c F^0_{\alpha ,-N}\in\c L^{(v_0)}_k$. 
\end{Thm}

\begin{Prop} \label{P5.2} 
 Theorem \ref{T5.1} holds with $\wt{N}(v_0)=N^*-1$. 
\end{Prop}

\begin{proof} 
Let $\c L_N^{\star }[v_0]$ be the minimal ideal in $\c L$ such that for all 
$\alpha\geqslant v_0$, $\c F^0_{\alpha ,-N}\in\c L_N^{\star }[v_0]_k$. 
We should prove that for 
$N\geqslant N^*-1$, $\c L_N^{\star }[v_0]=\c L^{(v_0)}$. 

Use 
induction on $s\geqslant 1$ to deduce 
(use that $\c F^0_{a,-N}\in\c L_N^{\star }[v_0]_k$) for $a\geqslant sv_0$   
that $D_{a0}\in\c L_N^{\star }[v_0]_k+\c L(s+1)_k$, 
e.g.\,cf.\,Lemma\,\ref{L3.5} or \,Lemma 2.3 from \cite{Ab12}. 
This implies that 
$\c L(p)\subset\c L_N^{\star }[v_0]$. 

Denote by $\bar{\c L}_N^{\star }[v_0]$ the image of 
$\c L_N^{\star }[v_0]$ in $\bar{\c L}=\c L/\c L(p)$. 

It follows from Theor.\,\ref{T4.1} that $\bar{\c L}[v_0]$ is already 
the minimal ideal in $\bar{\c L}$ such that $\{\c F^0_{\alpha [\iota ], -N}\ |
\ \iota\in\mathfrak{A}^0_1(p)\}\subset\bar{\c L}[v_0]_k$ 
(use that $m_{\iota }$ could be chosen  such that $m[\iota ]+m(\iota )=N\geqslant N^*-1$).
Therefore, it remains to show that for any $\alpha\geqslant v_0$, it holds 
$\bar{\c F}^0_{\alpha ,-N}\in\bar{\c L}[v_0]$. 

Note that $\bar{\c F}^0_{\alpha ,-N}\ne 0$ implies that $\alpha \in A[p-1,N]$. Then by 
Prop.\,\ref{P2.5}, $p^{N+1}(q\alpha -b^*)\geqslant q(q\alpha -(q-1)r^*)>(p-1)b^*$. 
Therefore, our proposition is implied by the following lemma. 
\end{proof}

\begin{Lem} \label{L5.3}
Suppose $M\geqslant 0$ and 
$p^{M+1}(q\alpha -b^*)>(p-1)b^*$. Then 
\medskip 

{\rm a)} $\bar{\c F}^0_{\alpha ,-M}\in\bar{\c L}[v_0]_k$;
\medskip 

{\rm b)} if in addition $p^{M}(q\alpha -b^*)>(p-1)b^*$ then 
$$\bar{\c F}^0_{\alpha ,-(M-1)}
\equiv \bar{\c F}^0_{\alpha ,-M}\,\op{mod}\,[\bar{\c L}[v_0],\bar{\c L}]_k\, .$$
\end{Lem}

\begin{proof}[Proof of lemma] 
Apply induction on $M\geqslant 0$. 
 
Let $M=0$, i.e. $p(q\alpha -b^*)>(p-1)b^*$.
Here $\bar{\c F}^0_{\alpha ,-M}=\bar {\c F}^0_{\alpha ,\,0}$ is a $k$-linear 
combination of the commutators 
$$[\dots [D_{a_1'0},D_{a_2'0}],\dots ,D_{a_r'0}]\, ,$$  
where 
$a_1'+\dots +a_r'=\alpha $. Consider two cases: 

(i)\ If $q\alpha -b^*>(p-1)b^*$ then 
$\alpha >pb^*/q>(p-1)(v_0-\delta _0)$ and this implies 
$\alpha \geqslant (p-1)v_0$, cf.\,Lemma\,\ref{L2.6}\,a). Therefore, 
the above commutators belong to 
$\bar{\c L}(p)=0$. (Indeed, if $(s_i-1)v_0\leqslant a_i'<s_iv_0$ then 
$\sum\limits _i\op{wt}(D_{a_i'\,0})=
\sum\limits _is_i>\alpha /v_0\geqslant p-1$.)

(ii)\ If $q\alpha -b^*\leqslant (p-1)b^*$ then   
$\iota :=q\alpha -b^*\in\mathfrak{A}^0_1(p)$ and $m(\iota )=0$. Then  
Theor.\,\ref{T4.1} implies  
$\bar{\c F}^0_{\alpha ,0}\in\bar{\c L}[v_0]_k$. 
\medskip 

 Suppose $M\geqslant 1$. We have the following two cases: 
 \medskip 

(i) If $qp^M\alpha -p^Mb^*\leqslant (p-1)b^*$ then there is $\iota\in\mathfrak{A}^0_1(p)$ 
and $n\geqslant 0$ such that 
$\iota p^n=qp^M\alpha -p^Mb^*\in\mathfrak{A}^0_1$ and, 
therefore, $m(\iota )=n\leqslant M$. 
Then by Theorem \ref{T4.1} with $m_{\iota }=M-n$ we obtain 
$\bar{\c F}^0_{\alpha ,-M}=\sigma ^{-M}
\bar{\c F}[\iota p^n]\in\bar{\c L}[v_0]_k$. 
 
(ii) Suppose now that  $p^M(q\alpha -b^*)>(p-1)b^*$. 
Prove simultaneously the remaining case of a) and the statement b). 

By the induction assumption 
$\bar{\c F}^0_{\alpha ,-(M-1)}\in\bar{\c L}[v_0]_k$. 

Note that 
$\bar{\c F}^0_{\alpha ,-M}-\bar{\c F}^0_{\alpha ,-(M-1)}$ 
is a linear combination of 
the terms of the form 
\begin{equation} \label{E5.1} 
[\dots [\bar{\c F}^0_{\alpha ', -(M-1)},D_{a_1',-M}],
\dots ,D_{a_r',-M}]\, ,
 \end{equation}
where $\alpha =\alpha '+(a_1'+\dots +a_r')/p^M$, 
$r\geqslant 1$ 
and $\alpha '\in A[p-1,M-1]$. 
\medskip 

It remains to prove that \eqref{E5.1} belongs to 
$[\bar{\c L}[v_0],\bar{\c L}]_k$. 
\medskip 

Let $s\in\N $ be such that  $sb^*/q>a_1'+
\ldots +a_r'\geqslant (s-1)b^*/q$. 
\medskip 

Then 
$a_1'+\ldots +a'_r\geqslant (s-1)v_0$, cf. Sect.\ref{S2.1}, and    
$\sum \limits _{i}\op{wt}(D_{a'_i,-M})\geqslant s$.  
 (If $(s_i-1)v_0\leqslant a_i'<s_iv_0$ then 
$\sum \limits_is_i>(a_1'+\ldots +a_r')/v_0\geqslant s-1$.) 
We can assume that $s\leqslant p-2$ because, otherwise, \eqref{E5.1} belongs to 
$\bar{\c L}(p)_k=0$. 
Now the inequality 
$a_1'+\ldots +a_r'<sb^*/q$ implies 
$$(p-1)b^*/p^M<q\alpha -b^*
\leqslant q\alpha '-b^*+sb^*/p^M\, $$
and, therefore, 
\begin{equation} \label{E5.2} 
p^M(q\alpha '-b^*)>(p-1-s)b^*\, .
\end{equation}
 
If $p^M(q\alpha '-b^*)>(p-1)b^*$ then by the  
induction assumption we have 
$\bar{\c F}^0_{\alpha ', -(M-1)}\in\bar{\c L}[v_0]_k$ and   
\eqref{E5.1} belongs to $[\bar{\c L}[v_0],\bar{\c L}]_k$. 
\medskip 

If  $p^M(q\alpha '-b^*)\leqslant (p-1)b^*$ then 
$\iota ':=p^M(q\alpha '-b^*)\in\mathfrak{A}^0_1$,  
$m(\iota ')=0$ (use that $\iota '>b^*$) and, therefore, 
$\bar{\c F}[\iota ']\in\bar{\c L}[v_0]_k$. 

Then from inequality \eqref{E5.2} 
and Remark from Sect.\,\ref{S4.6} it follows that  
$$\bar{\c F}[\iota ']_0\in [\bar{\c L}[v_0],\bar{\c L}]_k+
\bar{\c L}(p-s)_k\, .$$ 

Note that $\alpha '\in A[p-1,M-1]$ implies that  
$p^M\alpha '\equiv 0\,\op{mod}\,p$ and, therefore, $\iota '/p\in\mathfrak{A}^0_1$. 
Now from the the identity 
$\bar{\c F}[\iota ']=\bar{\c F}[\iota ']_0+\sigma \bar{\c F}[\iota '/p]$ 
we deduce 
that 
$\bar{\c F}[\iota '/p]\in\bar{\c L}[v_0]_k+\bar{\c L}(p-s)_k$,    
and  from  
$\sigma ^{M-1}\bar{\c F}^0_{\alpha ', -(M-1)}=\bar{\c F}[\iota '/p]$  
it follows that   
$\bar{\c F}^0_{\alpha ', -(M-1)}\in\bar{\c L}[v_0]_k+\bar{\c L}(p-s)_k$. 
 
Finally, the commutator 
$$[\ldots [\bar{\c L}(p-s)_k, D_{a_1',-M}],\dots ,D_{a'_r,-M}]
\subset\bar{\c L}(p)=0$$
because $\sum\limits _{i}\op{wt}(D_{a_i',-M})\geqslant s$. 
As a result,  \eqref{E5.1} belongs to $[\bar{\c L}[v_0],\bar{\c L}]_k$. 

The lemma is proved. 
\end{proof} 

\begin{remark}
 If in the above notation 
 $p^M(q\alpha -b^*)>(p-1)b^*$, then 
 $$\c F^0_{\alpha ,-(M-1)}
\equiv \c F^0_{\alpha ,-M}\,\op{mod}\,[\c L[v_0],\c L]_k\, .$$

Indeed, since $\c F^0_{\alpha ,-M}$ and $\c F^0_{\alpha ,-M+1}$ 
have the same linear term the part b) of the above lemma implies that 
$$\c F^0_{\alpha ,-(M-1)}-
\c F^0_{\alpha ,-M}\in \c L(p)_k\cap C_2(\c L_k)+
[\c L[v_0],\c L]_k\, .$$
It remains to note that Lemma\,\ref{L3.5} implies 
$\c L(p)\cap C_2(\c L)\subset [\c L^{(v_0)},\c L]$. 
\end{remark}

\subsection{Flexible boundaries} \label{S5.2} 
Suppose $v\geqslant 1$. Introduce the weight function $\op{wt}_v$ on $\c L_k$ 
such that $\op{wt}_v(D_{an})=s\in\N $ iff $(s-1)v\leqslant a<sv$. 
Denote by $\c L_v(p)$ the ideal of elements with 
$\op{wt}_v$-weight $\geqslant p$. Note that in notation from Sect.\ref{S1.6} 
$\op{wt}=\op{wt}_{v_0}$.

Introduce another weight function $\op{wt}_v^{+}$ on $\c L_k$ such that 
$\op{wt}^{+}(D_0)=1$ and for $a\in\Z ^+(p)$, $\op{wt}_v^+(D_{an}=s$ iff 
$(s-1)v<a\leqslant sv$. Denote by $\c L_v^+(p)$ the ideal of elements with 
$\op{wt}_v^+$-weight $\geqslant p$.

Clearly, we have the following property: 

\begin{Prop} \label{P5.4} 
 $\c L_v^{+}(p)={\cup} _{v'>v}\c L_{v'}(p)$. 
\end{Prop}

Suppose $v>1$ and $v^{\flat }\in [1,v)$ is such that for any $v'\in (v^{\flat }, v]$, 
we have $\c G_{<p}^{(v')}=\c G_{<p}^{(v)}$. The existence of $v^{\flat }$ 
follows from the left-continuity 
property of ramification filtration.  

\begin{remark} There is the following upper estimate for $v^{\flat }$. 
 Let  $\mathfrak{B}$ 
 be the set of all  
 $a_1+a_2p^{-n_2}+\ldots a_{p-1}p^{-n_{p-1}}<v$ with   
 $a_i\in\Z ^+(p)\cap [1,(p-1)v)$ and $n_i\geqslant 0$. 
 Let $\delta _0(1)=\min\,\{v-b\,|\ b\in\mathfrak{B}\}$, cf. Sect.\,\ref{S2.1}. 
 Then $v^{\flat }\leqslant v-\delta _0(1)$. This follows easily from 
 Theorem \ref{T5.1} because if $\alpha\notin\mathfrak{B}$ and $\alpha <v$ 
 then 
 $\c F^0_{\alpha ,-N}=0$ and, therefore, 
 the set $\mathfrak{B}$ contains all possible breaks 
 of the  filtration $\{\c G_{<p}^{(v')}\}_{1\leqslant v'<v}$. 
\end{remark}

For any $\alpha > v^{\flat }$ choose 
$N_{\alpha }\geqslant 0$ such that 
$$p^{N_{\alpha }+1}(\alpha -v^{\flat })>(p-1)v^{\flat }\, .$$ 

There is the following more effective   
version of Theor.\,\ref{T5.1}.

\begin{Thm}\label{T5.5} 
$\c L^{(v)}$ is the minimal ideal in $\c L$ 
such that 
for all $\alpha\geqslant v$, $\c F^0_{\alpha ,-N_{\alpha }}\in\c L_k^{(v)}$.  
\end{Thm}

\begin{proof} 
Apply Theorem \ref{T5.1} with $v_0=v$ by choosing 
$N\geqslant \wt{N}(v)$ such that 
for all $\alpha\geqslant v$, $p^{N+1}(\alpha -v^{\flat })>(p-1)v^{\flat }$. 
Then $\c L^{(v)}$ is the minimal ideal in $\c L$ such that 
 $\c L^{(v)}_k$ contains 
all $\c F_{\alpha ,-N}^0$ with $\alpha\geqslant v$.

Fix $\alpha\geqslant v$ and choose 
$v_{\alpha }\in (v^{\flat },v)$ such that we still have the inequalities   
$p^{N_{\alpha }+1}(\alpha -v_{\alpha })>(p-1)v_{\alpha }$ and 
$p^{N+1}(\alpha -v_{\alpha })>(p-1)v_{\alpha }$

Let $b^*_{\alpha }$ and $q_{\alpha }$ be analogs of the parameters $b^*$ and $q$ 
chosen in Sect.\ref{S2.1} when $v_0$ is replaced by $v_{\alpha }$. 

Then the inequality $p^{M+1}(q_{\alpha }\alpha -b_{\alpha }^*)>(p-1)b_{\alpha }^*$ 
from Lemma \ref{L5.3} holds with $M=N$ and $M=N_{\alpha }$ 
(use that $b_{\alpha }^*/q_{\alpha }<v_{\alpha }$). 
Therefore, by remark from the end of Sect.\,\ref{S5.1} we have 
$$\c F^0_{\alpha ,-N}\equiv \c F^0_{\alpha ,-N_{\alpha }}\,\op{mod}\,
[\c L^{(v)},\c L]_k\, .$$
This means that the conditions 
$\c F^0_{\alpha ,-N}\in\c L^{(v)}_k$ and 
$\c F^0_{\alpha ,-N_{\alpha }}\in\c L^{(v)}_k$ 
are equivalent. 
Theorem is proved. 
\end{proof}

\subsection{The whole filtration 
$\{\c G^{(v)}_{<p}\ |\ v\geqslant 1\}$} \label{S5.3} \ \ 

Suppose 
$1=v_1<v_2<\ldots <v_r<\ldots $  
are all jumps of the ramification 
filtration $\{\c G_{<p}^{(v)}\}_{v\geqslant 1}$. (This set is obviously discrete.) 
In other words, 

-- $\c G_{<p}^{(v_1)}\varsupsetneqq \ldots \varsupsetneqq \c G_{<p}^{(v_r)}
\varsupsetneqq \ldots $; 
\medskip 

-- $\c G^{(1)}_{<p}$ is the ramification subgroup in $\c G_{<p}$, 
$(\c G_{<p}:\c G^{(1)}_{<p})=p$;
\medskip 

-- if $r\geqslant 2$ and $v_{r-1}<v\leqslant v_r$ 
then $\c G^{(v)}_{<p}=\c G_{<p}^{(v_r)}$.
\medskip 

Use the identification $\c G_{<p}\simeq G(\c L)$ from Sect.\,\ref{S1}. 
Then the ramification filtration appears as 
ideals 
$\c L^{(v_1)}\varsupsetneqq \ldots \varsupsetneqq \c L^{(v_r)}
\varsupsetneqq \ldots $ in $\c L$, 
where $\c L^{(1)} _k$ is 
generated by all $D_{an}$, $a\in\Z ^+(p)$. 
\medskip 

Suppose $u\geqslant 2$. 

Introduce the weight function $\op{wt}_u$ on $\c L_k$ such that 
$\op{wt}_u(D_0)=1$ and if $s\in\N $ is such that 
$(s-1)v_{u-1}<a\leqslant sv_{u-1}$ then $\op{wt}_u(D_{an})=s$.

Introduce also the elements 
$\c F^*[u]\in\c L_k$ obtained from the elements 
$\c F^0_{v_u ,-M_u}$, cf. Sect.\,\ref{S4.2}, where 
$$p^{M_u+1}(v_u-v_{u-1})>(p-1)v_{u-1}$$ 
by imposing additional restriction 
$\op{wt}_u(D_{a_1n_1})+\ldots +\op{wt}_u(D_{a_sn_s})\leqslant p-1$ if $s\geqslant 2$. 
Clearly, 
\begin{equation} \label{E5.3} 
\c F^*[u]\equiv \c F^0_{v_u, -M_u}\,\op{mod}\, [\c L^{(v_u)},\c L]_k
\end{equation}

\begin{Thm}\label{T5.6}
 For $r\geqslant 2$, $\c L^{(v_r)}$ is the minimal ideal in $\c L$ 
such that $\c L^{(v_r)}_k$ contains all $\c F^*[u]$ with $u\geqslant r$. 
\end{Thm}

\begin{proof} 
Consider $\alpha >1$ and 
let $u_{\alpha }\geqslant 2$ be such that 
$v_{u_{\alpha }-1}<\alpha\leqslant v_{u_{\alpha }}$. Let  
$N_{\alpha }\geqslant 0$ be such that 
$p^{N_{\alpha }+1}(\alpha -v_{u_{\alpha }-1})>(p-1)v_{u_{\alpha }-1}$. 
This implies that $\c F_{\alpha ,-N_{\alpha }}\in \c L_k^{(v_{u_{\alpha }})}$. 

Suppose $\alpha > v_{r-1}$. Then  $v_{u_{\alpha }-1}\geqslant v_{r-1}$ and  
$$p^{N_{\alpha }+1}(\alpha -v_{r-1})
\geqslant p^{N_{\alpha }+1}(\alpha -v_{u_{\alpha }-1})>(p-1)v_{u_{\alpha }-1}
\geqslant (p-1)v_{r-1}\, .$$

In particular,  Theor.\,\ref{T5.5} implies that $\c L^{(v_r)}$ 
is the minimal ideal in $\c L$ 
such that for all $\alpha\geqslant v_r$, 
$\c F^0_{\alpha ,-N_{\alpha }}\in\c L^{(v_r)}_k$. 

If $\alpha >v_r$ 
then $u_{\alpha }\geqslant v_{r+1}$ and 
$\c F^0_{\alpha ,-N_{\alpha }}\in\c L^{(v_{r+1})}_k$. 

If $\alpha =v_r$ then $u_{\alpha }=r$ and $p^{N_{\alpha }+1}(v_r-v_{r-1})=
p^{N_{\alpha }+1}(\alpha -v_{u_{\alpha }-1})>(p-1)v_{u_{\alpha }-1}=
(p-1)v_{r-1}$. Set $M_r=N_{\alpha }$. Then congruence \eqref{E5.3} implies that  
$$\c F^*[r]\equiv \c F^0_{v_r,-M_{r}}\equiv \c F^0_{\alpha , -N_{\alpha }}
\,\op{mod}\,[\c L^{(v_r)},\c L]_k\, .$$ 

As a result, $\c L^{(v_r)}$ is the minimal ideal in $\c L$ 
such that $\c F^*[r]\in\c L^{(v_r)}_k$ and 
$\c L^{(v_{r+1})}\subset \c L^{(v_r)}$. 

By iterating this procedure we obtain the statement of our theorem. 
\end{proof}

\end{document}